\documentclass[a4paper,11pt]{article} 
\usepackage{amsmath}
\usepackage[retainorgcmds]{IEEEtrantools}
\usepackage{graphicx}
\usepackage[margin=1in]{geometry}
\usepackage{amsfonts}
\usepackage{amsthm}
\usepackage{setspace}
\usepackage{mathtools}
\usepackage{listings}
\usepackage{mathabx}
\usepackage{amssymb}
\usepackage[mathscr]{euscript}
\usepackage{pictexwd,dcpic}
\newtheorem{thm}{Theorem}
\theoremstyle{definition}
\newtheorem{lem}{Lemma}
\newtheorem{dfn}{Definition}

\newtheorem{crl}{Corollary}

\newtheorem{cnj}{Conjecture}

\newcommand{\N}{\mathbb{N}}

\newcommand{\R}{\mathbb{R}}

\newcommand{\ord}{\text{ord}}
\newcommand{\scr}[1]{\mathscr{#1}}
\author{Cole Hugelmeyer} 
\title{A Representation Theorem for Knots and a Generalization of the Fundamental Theorem of Finite Type Invariants}
\begin{document} 
\maketitle 

\abstract{We provide a way to produce knots in $S^3$ from signed chord diagrams, and prove that every knot can be produced in this way. Using these diagrams, we generalize the fundamental theorem of finite type invariants. We also provide moves for the diagrams so that any two diagrams for the same knot are connected by a sequence of moves.}

\section{Introduction}
The first part of this paper is devoted to describing a combinatorial method of representing knots with signed chord diagrams. Roughly speaking, we produce a knot where the chords in the chord diagram correspond to clasps in the knot, and the clasps lie on top of each other in an order which is compatible with the cyclic order of the diagram. Our first main result, Theorem \ref{rep}, is that all knots can be represented by these diagrams.

After seeing some examples of these diagrams, we will use them to generalize the fundamental theorem of finite type invariants. The fundamental theorem of finite type invariants gives a linear presentation for the space $\scr{V}_n/\scr{V}_{n-1}$ of rank $n$ finite type invariants modulo the subspace of rank $n-1$ finite type invariants. We generalize this result by giving a linear presentation for the space of finite type invariants $\scr{V}_n$  directly. This presentation extends the standard presentation for $\scr{V}_n/\scr{V}_{n-1}$ by adding higher order terms. This is our second main result, Theorem \ref{gen}.

We will give two proofs for the fact that all knots can be represented by signed chord diagrams. One which is geometric and visual, and one which is more formal and combinatorial. Using the former method, we will include the addition of bands on the knot into the representation so that concepts like unknotting number and band surgery can be shown to behave nicely with respect to signed chord diagrams. Using the latter method, we will give a formal procedure for transforming a knot diagram into a signed chord diagram, and this will allow us to analyze how Reidemeister moves in a knot diagram affect the resulting signed chord diagram. We will describe a set of moves on signed chord diagrams which preserve the knot type. Our final result, Theorem \ref{moves}, states that any two signed chord diagrams for the same knot   are related by a sequence of moves.

\section{The Construction}

We begin by describing a way of producing banded unknots from chord diagrams. 

\begin{dfn}
A chord diagram is a cyclically ordered finite set $D$ equipped with a perfect matching $\tau: D\to D$, an involution with no fixed points. The order of a chord diagram, $\ord(D)$, is defined to be $|D/\tau|$. A subdiagram of a chord diagram is a subset $D'\subseteq D$ which is closed under $\tau$ and equipped with the same cyclic ordering as $D$. We will think of a chord diagram as a drawing in a disk with points on the boundary representing the elements of $D$, and lines between the points representing the elements of $D/\tau$, which we call chords.
\end{dfn}

\begin{dfn}
A banded knot is a smooth knot $\gamma: S^1\to S^3$ equipped with a finite set of smoothly embedded bands $\beta_i: [0,1]^2\to S^3$, $i\in F$, such that $$\{\beta_i([0,1]\times\{s\}): s\in \{0,1\}, i\in F\}$$ is a set of disjoint intervals in the image of $\gamma$ whose union is $$\gamma(S^1)\cap \bigcup_{i \in F} \beta_i([0,1]^2).$$ A banded knot is said to be orientable if for all $i$, the point $\beta_i(x,s)$ moves in the positive direction on $\gamma(S^1)$ exactly when either $x$ increases and $s = 0$, or $x$ decreases and $s = 1$.

Let $\rho: [0,1]^2\to [0,1]^2$ denote an orientation preserving 180 degree rotation of the square. We will say two banded knots $(\gamma,\{\beta_i\}_{i\in F})$ and $(\gamma',\{\beta'_i\}_{i\in F'})$  are isotopic if there is an orientation preserving diffeomorphism $\phi: S^3\to S^3$, a function $s:F\to \{0,1\}$, a bijection $\sigma:F\to F'$, and an orientation preserving diffeomorphism $p: S^1\to S^1$ such that $\phi\gamma p = \gamma'$, and $\phi\beta_{i}\rho^{s(i)} = \beta'_{\sigma(i)}$ for all $i\in F$.
\end{dfn}

\begin{dfn}
Given a chord diagram $D$, we construct an orientable banded unknot, $B(D)$, which is well-defined up to isotopy, in the following way. First, we select a smooth unknot $\gamma: S^1\to S^3$ and we fix a smooth map $\phi: D^2\times S^1\to S^3$ such that $$\phi|_{\text{Int}(D^2\times S^1)}: \text{Int}(D^2\times S^1)\to S^3\setminus \gamma(S^1)$$ is an orientation reversing diffeomorphism, and such that there exists an orientation preserving diffeomorphism $\iota: S^1\to \partial D^2$ so that for all $x,y\in S^1$, we have $\phi(\iota(x),y) = \gamma(x)$.  Now, we choose a injective function $f: D\to S^1$ so that the cyclic ordering on $D$ induced by the orientation on $S^1$ is the same as the cyclic ordering with which $D$ was already equipped. Now we choose a set of disjoint closed intervals in $S^1$, $\{I_i\}_{i\in D}$, so that $f(i)$ is in the interior of $I_i$. We now select a base point $*\in S^1$ which is not in the image of $f$. Moving in the positive direction in $S^1$ from the base point induces a linear ordering ``$<_*$'' on $D$ which is consistent with the cyclic ordering. We then let $D_* = \{x\in D: x <_* \tau(x)\}$.  Now, for each $i\in D_*$, we select a smooth orientation preserving embedding $q_i: [0,1]^2\to D^2$ such that $q_i([0,1]^2)\cap\partial D^2 = \iota(I_i)\cup \iota(I_{\tau(i)})$ and $q_i([0,1]\times \{0\}) = \iota(I_i) $, and $q_i([0,1]\times \{1\}) = \iota(I_{\tau(i)})$. Finally, we produce a banded knot $B(D) := (\gamma, \{\beta_i\}_{i\in D_*})$, where for any $i\in D_*$ and $x\in [0,1]^2$, we have $\beta_i(x) := \phi(q_i(x), f(i))$. 

\begin{figure}[h]
\caption{A chord diagram $D$ and the banded unknot $B(D)$.}
\centering
\includegraphics[scale = 0.6]{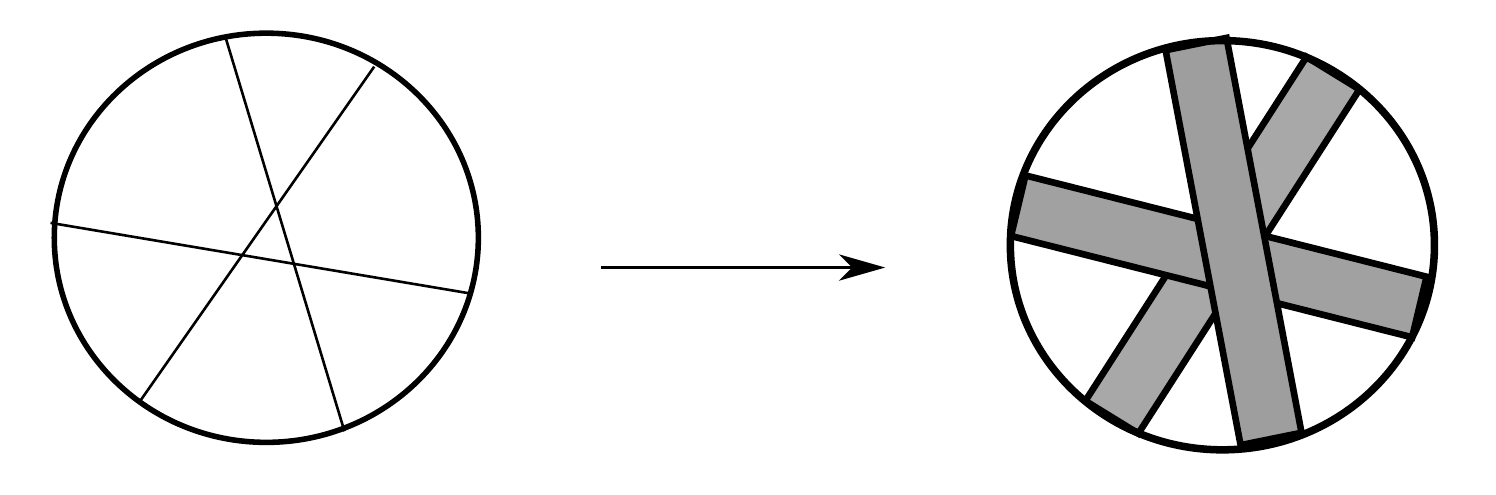}
\end{figure}

We have now constructed a banded knot from a chord diagram, but our construction depends on a choice of base point. For this definition to make sense, we need the following lemma. 
\end{dfn}

\begin{lem}
The isotopy type of $B(D)$ does not depend on the choice of base point $*\in S^1$. 
\end{lem}

\begin{proof}
It suffices to show that if we move $*$ in the positive direction over a point $f(i)$, then the resulting banded knot is isotopic to what we had before. We will have two different base points $*$ and $*'$, where the former is just before $f(i)$ in $S^1$ and the latter is just after $f(i)$. This gives us two different sets $D_*$ and $D_{*'}$ as in the construction. The only difference in the ordering is that $i$ is the first element of $D$ with respect to $<_*$ but not with respect to $<_{*'}$. In particular, $i <_* \tau(i)$ but $\tau(i) <_{*'} i$. This means the only modification to $B(D)$ is that the band $\beta_{i}(x) = \phi(q_i(x), f(i))$ is replaced with a band $\beta_{\tau(i)}(x) = \phi(q_{\tau(i)}(x), f(\tau(i)))$. It can be arranged that $q_i = q_{\tau(i)}\rho$. Let $h: [0,1]\to S^1, h(1) = f(\tau(i)), h(0) = f(i)$ parameterize the interval in $S^1$ between $i$ and $\tau(i)$ that contains $*$. Furthermore, if $\tau(i) <_* j$ for some $j\in D_*$, we may select $q_j([0,1]^2)$ to be disjoint from $q_i([0,1]^2) = q_{\tau(i)}([0,1]^2)$. Putting all of this together, we can continuously perturb the band $\beta_i$ to the band $\beta_{\tau(i)}\rho$ via $\beta(x,t) = \phi(q_i(x), h(t))$, and this perturbation can be arranged to not induce self-intersections of the banded knot. Therefore, this perturbation induces an isotopy. 
\end{proof}

It should be noted at this point that it is critical for our definition of isotopy of banded knots to include 180 degree rotations of the parameterizations of the bands, because as can be seen from this lemma, there is no preferred edge for any of the bands in $B(D)$. The reason we have chosen $\phi|_{\text{int}(D^2\times S^1)}$ to be orientation reversing rather than orientation preserving is that it will make it easier to draw the banded knots that we will construct. 

\begin{dfn}\label{surgery}
Let $X = (\gamma,\{\beta_i\}_{i\in F})$ be an oriented banded knot. We describe two ways of modifying this knot at a band $i\in F$, which we call the $+1$ and $-1$ clasp surgery. Let $s \in \{1,-1\}$. First, we ``thicken'' our band. That is to say, we select an orientation preserving embedding $\theta: [0,1]^3\to S^3$ which intersects our banded knot exactly at $\beta_i$, and has $\beta_i(x) = \theta(x,1/2)$ for all $x\in [0,1]$. Let $I_0 = \theta([0,1]\times \{0\}\times\{1/2\})$ and $I_1 = \theta([0,1]\times \{1\}\times\{1/2\})$ be the two intervals in $\partial(\theta([0,1]^3))$ at which $\beta_i$ coincides with $\gamma$. Let $\eta_0,\eta_1:[0,1]\to [0,1]^3$ be given by $$\eta_0(t) = \left(t,3t\left(1-t\right),\frac{1}{2} + st\left(\frac{1}{2}-t\right)\left(1-t\right)\right)$$  $$\eta_1(t) = \left(1-t,1-3t\left(1-t\right),\frac{1}{2} + st\left(\frac{1}{2}-t\right)\left(1-t\right)\right)$$ We now modify our banded knot. We remove the band $\beta_i$, and we replace $I_0$ and $I_1$ with $\theta\eta_0([0,1])$ and $\theta\eta_1([0,1])$ respectively. We then smooth the corners. The resulting banded knot is called the $s$ clasp surgery of $X$ at $i$.
\end{dfn}

\begin{figure}[h]
\caption{Clasp surgery on a band.}
\centering
\includegraphics[scale = 0.7]{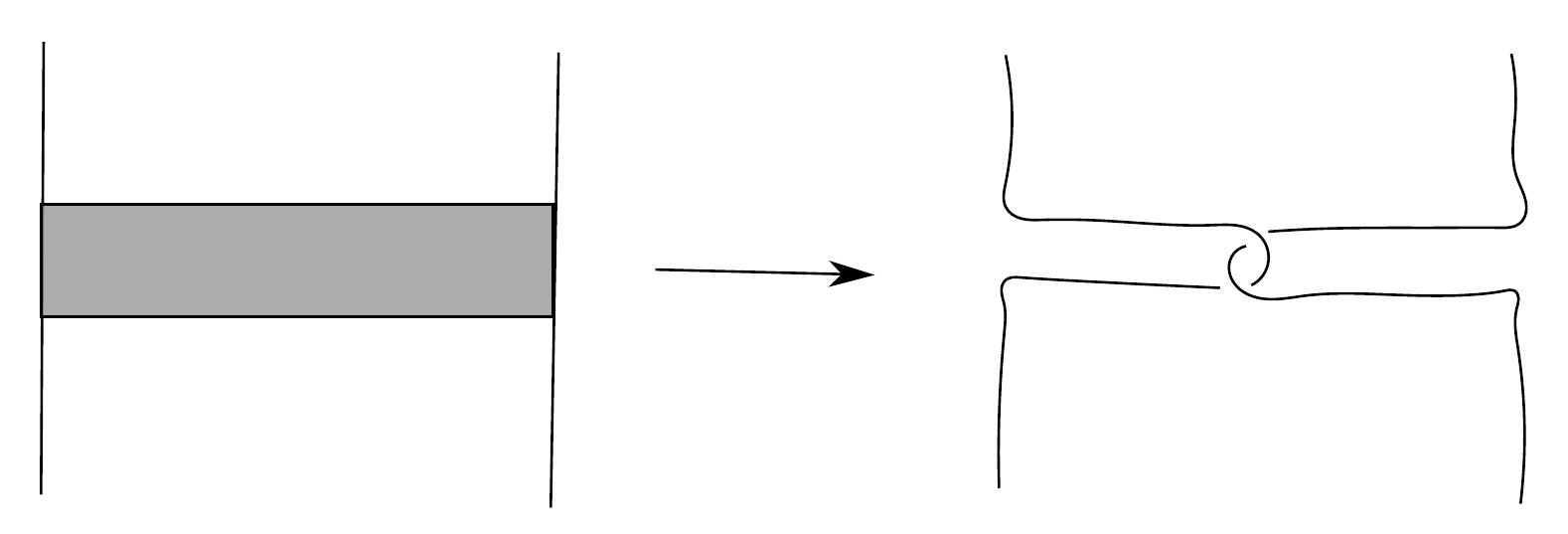}
\end{figure}

\begin{dfn}
A partially signed chord diagram is a cyclically ordered set $D$ with a perfect matching $\tau: D\to D$ and a function $s: D/\tau \to \{-1,0,1\}$. A signed chord diagram is a partially signed chord diagram for which $0$ is not in the image of $s$. Given a partially signed chord diagram, $D$, we define its geometric realization to be the oriented banded knot given by applying $s(i)$ clasp surgery on $B(D)$ for each $i\in D/\tau$ with $s(i)\neq 0$. The geometric realization of a partially signed chord diagram is denoted $\Gamma(D)$. Note that the geometric realization of $D$ is a knot if and only if $D$ is a signed chord diagram. 
\end{dfn}

We can now state our main theorem.

\begin{thm}\label{rep}
Every orientable banded knot is isotopic to the geometric realization of some partially signed chord diagram. 
\end{thm}

We will postpone the proof of this theorem to Section \ref{main}. In the meantime, we will examine the consequences of this result. 

\begin{crl}\label{knot}
Every oriented knot is representable as the geometric realization of some signed chord diagram. 
\end{crl}
\proof{Oriented knots are oriented banded knots with no bands.}

\begin{dfn}
Given a banded knot, $X$, let $\Lambda(X)$ denote the surgery of $X$ where one removes the interior of the bands, and the interior of the intervals where the bands coincide with the knot. This produces an oriented link.
\end{dfn}

\begin{crl}
For every oriented link $L$, there is a partially signed chord diagram $D$ so that $\Lambda(\Gamma(D))$ is isotopic to $L$. Moreover, the chords of sign zero can be selected to not cross one another
\end{crl}
\begin{proof}
Given a link, we can connect all of the link components together by bands to produce a banded knot. We can connect the components sequentially, and the chords representing bands will all be parallel. 
\end{proof}

\begin{crl}
A knot $K$ is ribbon if and only if there is some signed chord diagram $D$ with $\Gamma(D)$ isotopic to $K$, and a diagram $D'$ which is made from $D$ by inserting non intersecting bands and has $\Lambda(\Gamma(D'))$ isotopic to the trivial link. 
\end{crl}

\begin{proof}
If we add $n$ bands to a knot so that the surgery on these bands gives an $n+1$ component link, then the cord diagram for the bands must have no intersecting chords. This means that for any ribbon knot we can find a pair of diagrams as described. Of course, if such diagrams exist then the knot is ribbon. 
\end{proof}

\begin{crl}
A knot $K$ has unknotting number $\leq n$ if and only if there is a signed chord diagram $D$ and a subdiagram $D'\subseteq D$, such that the difference between the order of $D$ and the order of $D'$ is at most $n$, $\Gamma(D)$ is isotopic to $K$, and $\Gamma(D')$ is an unknot. 
\end{crl}

\begin{proof}
The crossing changes required to transform an unknot into $K$ can be thought of as clasp surgeries on bands. Therefore, there is a banded unknot with at most $n$ bands, such that after performing clasp surgeries of some sign on each band, a knot isotopic to $K$ is obtained. Theorem \ref{rep} then gives us the desired result. 
\end{proof}

\begin{dfn}
Recall Definition \ref{surgery}. We can define a new kind of surgery, which we call the $s$ singular clasp surgery, by instead using the parameterization $$\eta_0(t) = \left(t,2t\left(1-t\right),\frac{1}{2} + st\left(\frac{1}{2}-t\right)\left(1-t\right)\right)$$  $$\eta_1(t) = \left(1-t,1-2t\left(1-t\right),\frac{1}{2} + st\left(\frac{1}{2}-t\right)\left(1-t\right)\right)$$ This produces a transverse double point. Given a banded knot $X = (\gamma,\{\beta_i\}_{i\in F})$, and a function $s: F\to \{1,-1\}$ we define $\Sigma_f(X)$ to be the singular knot obtained by applying s singular clasp surgeries to each band of $X$.
\end{dfn}

\begin{crl}\label{singular}
Let $K$ be a singular knot with transverse double points, and let $S$ be the set of double points. For any function $f:S\to \{1,-1\}$, there exists a partially signed chord diagram $D$, and a function $f': F\to \{1,-1\}$, where $\Gamma(D) = (\gamma, \{\beta_i\}_{i\in F})$ such that $K$ is isotopic to $\Sigma_{f'}(\Gamma(D))$ in such a way that if the isotopy takes a double point $x$ to a double point at which the corresponding band index is $i\in F$, then we have $f'(i) = f(x)$.
\end{crl}

\begin{proof}
Locally, a transverse self intersection can be represented as a singular clasp surgery of either sign. We can then represent our singular knot as a banded knot with signed bands. Theorem \ref{rep} then gives us the desired partially signed chord diagram, along with a function from the bands of its geometric realization to $\{1,-1\}$ which coincides with our choice of sign for the singular clasp surgeries. 
\end{proof}

\begin{dfn}
For an oriented knot $K$, let $A(K)$ denote the minimum order of all chord diagrams $D$ for which $\Gamma(D)$ is isotopic to $K$.
\end{dfn}

\begin{crl}
The Heegaard genus of $K$ is at most $A(K) + 1$. 
\end{crl}

\begin{proof}
If $D$ is a signed chord diagram. The boundary of a tubular neighborhood of $B(D)$ is a Heegaard splitting for the complement of $\Gamma(D)$. The genus of this Heegaard splitting is one more than the order of $D$.
\end{proof}
\newpage
\section{Some Examples}

We will now describe a procedure for actually drawing a knot diagram for the knot corresponding to a signed chord diagram. First, we adopt the convention that in a signed chord diagram, thin black lines are +1 chords, dashed lines are -1 chords, and thick grey lines are 0 chords. 

\begin{figure}[h]
\caption{A partially signed chord diagram.}
\centering
\includegraphics[scale = 0.45]{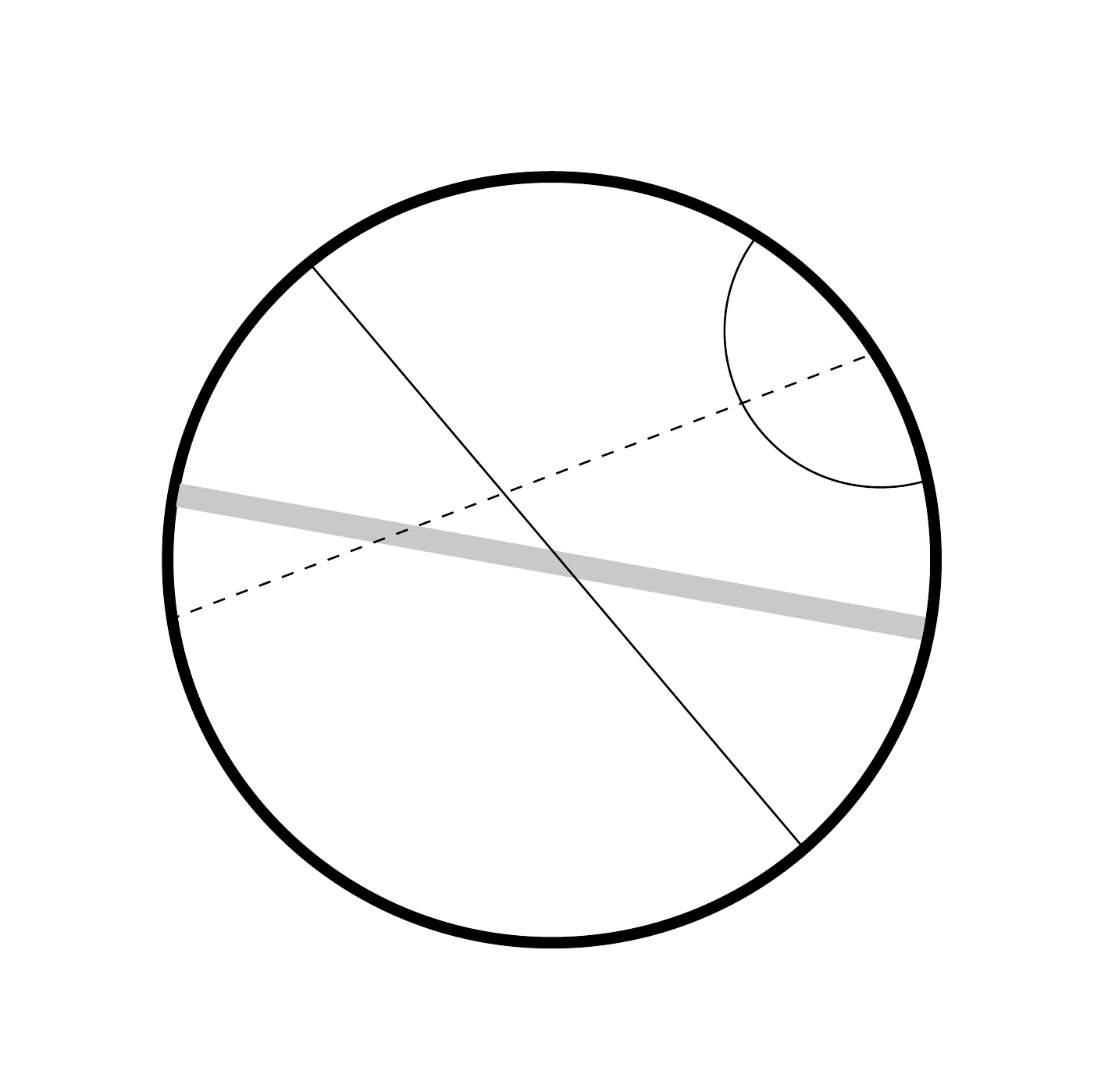}
\end{figure}

To draw a knot from a partially signed chord diagram, first draw the diagram on a line. It doesn't matter which point the diagram is expanded around.

\begin{figure}[h]
\caption{The above diagram drawn on a line.}
\centering
\includegraphics[scale = 0.5]{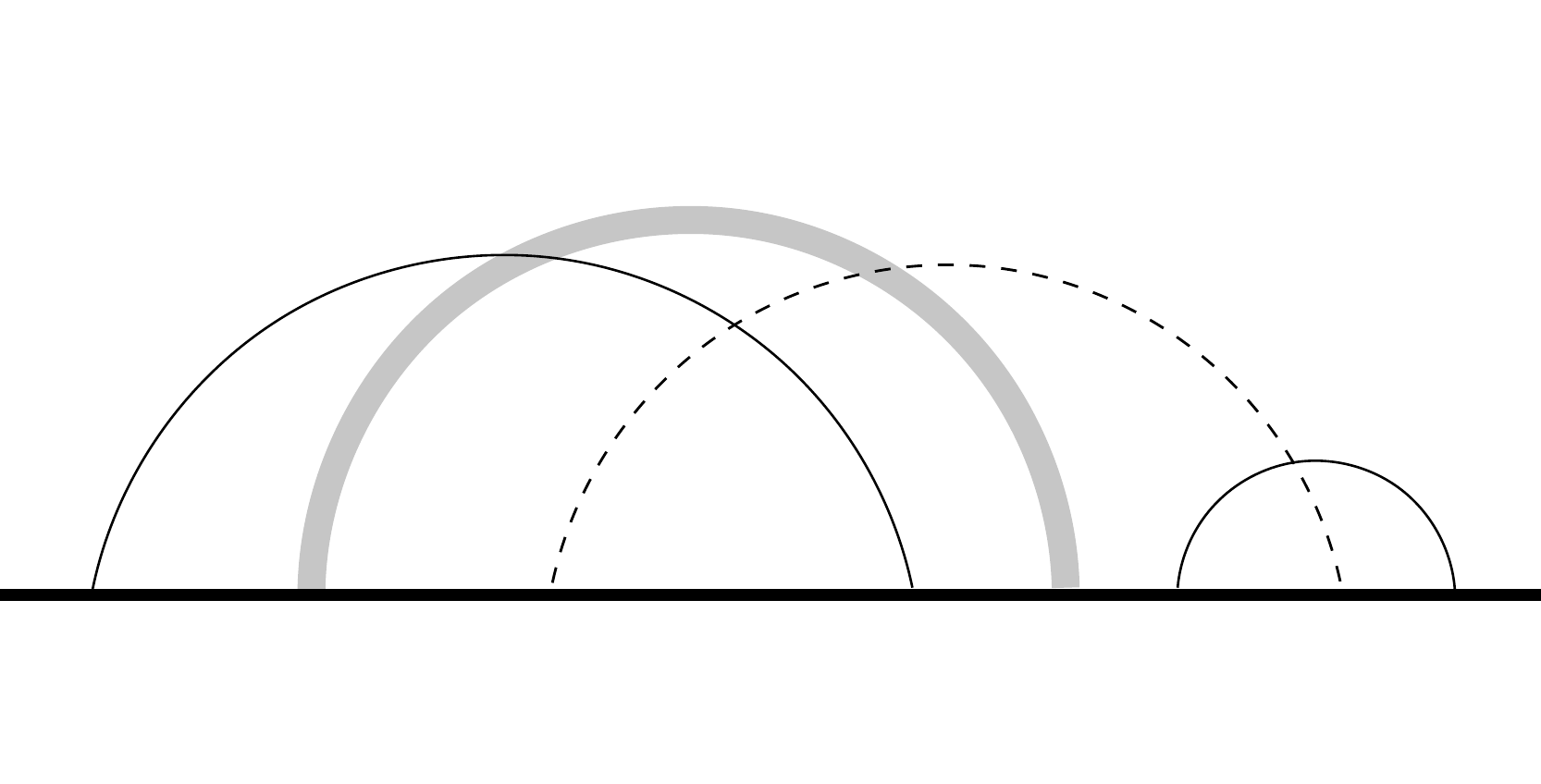}
\end{figure}

Then, move along the line. Every time we reach a point which is the first part of a chord, draw a gap unless the chord is sign zero. Every time we reach a point which is the second part of a chord, draw a clasp, going under any lines that are already drawn, which goes up and over to link with the corresponding gap via a clasp of the correct sign. For a band, draw a finger that goes back to the first part of the chord, and a small rectangular band.

\begin{figure}[h]
\caption{The knot we get from the above diagram.}
\centering
\includegraphics[scale = 0.5]{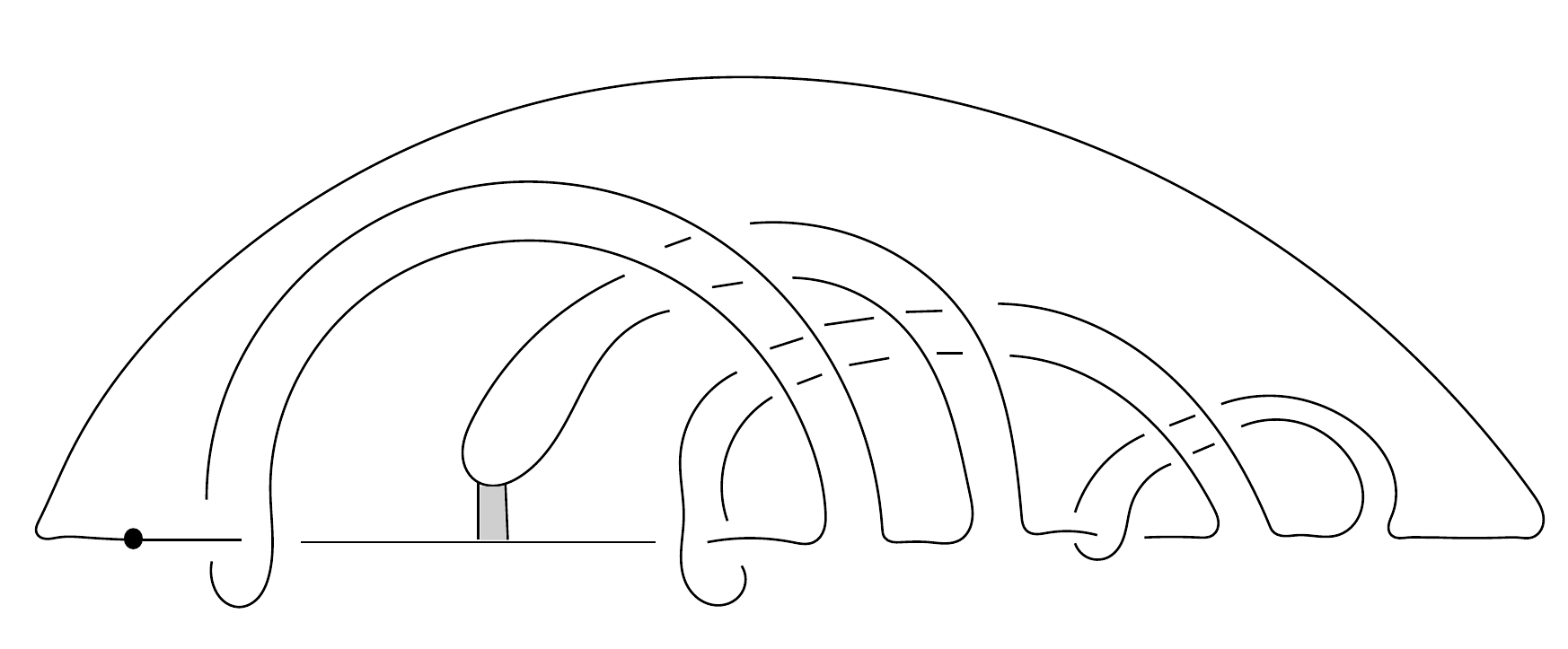}
\end{figure}

Now, we will give examples of signed chord diagrams for some common knots. 

\begin{figure}[h]
\caption{Signed chord diagrams and their corresponding knots.}
\centering
\includegraphics[scale = 0.5]{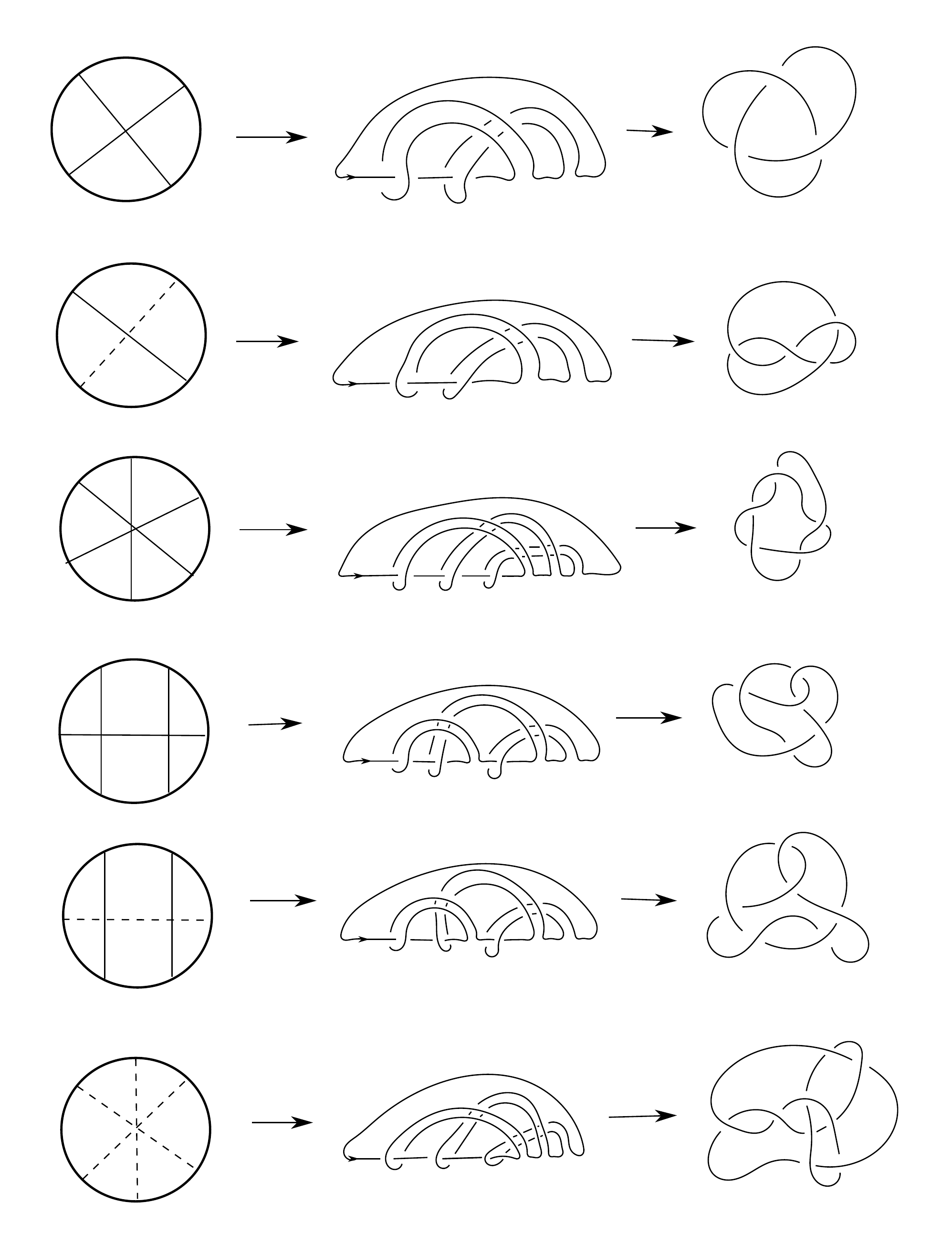}
\end{figure}

Note that in some cases, reversing the sign of every chord yields the mirror image knot. This is not true in general, however. As can be seen above, the diagram with three mutually intersecting +1 chords yields the right handed cinquefoil, but the diagram with three mutually intersecting -1 chords yields the knot $9_{48}'$. 

\begin{figure}[h]
\caption{Two diagrams for the cinquefoil, and three diagrams for the unknot.}
\centering
\includegraphics[scale = 0.4]{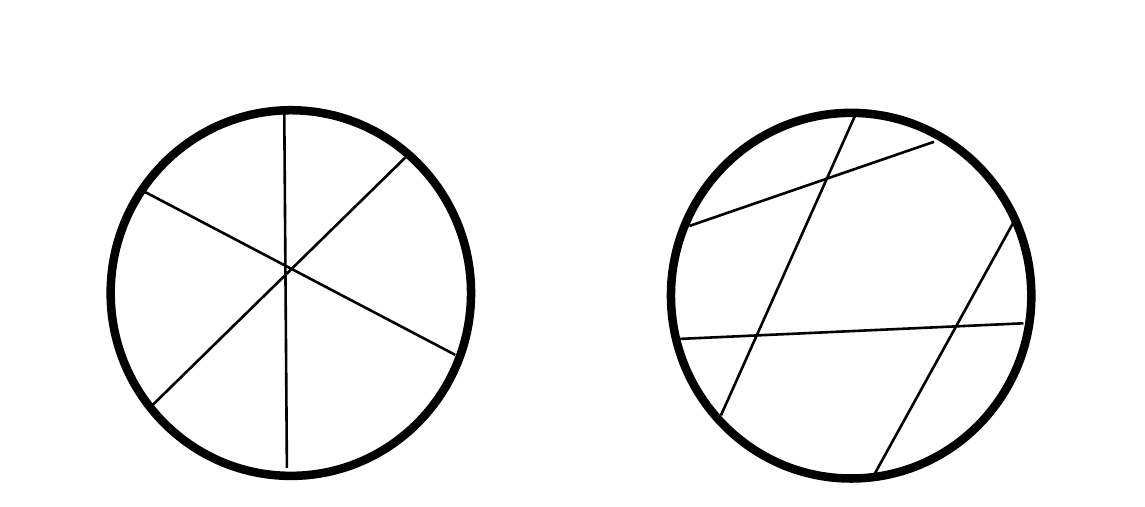}

\includegraphics[scale = 0.4]{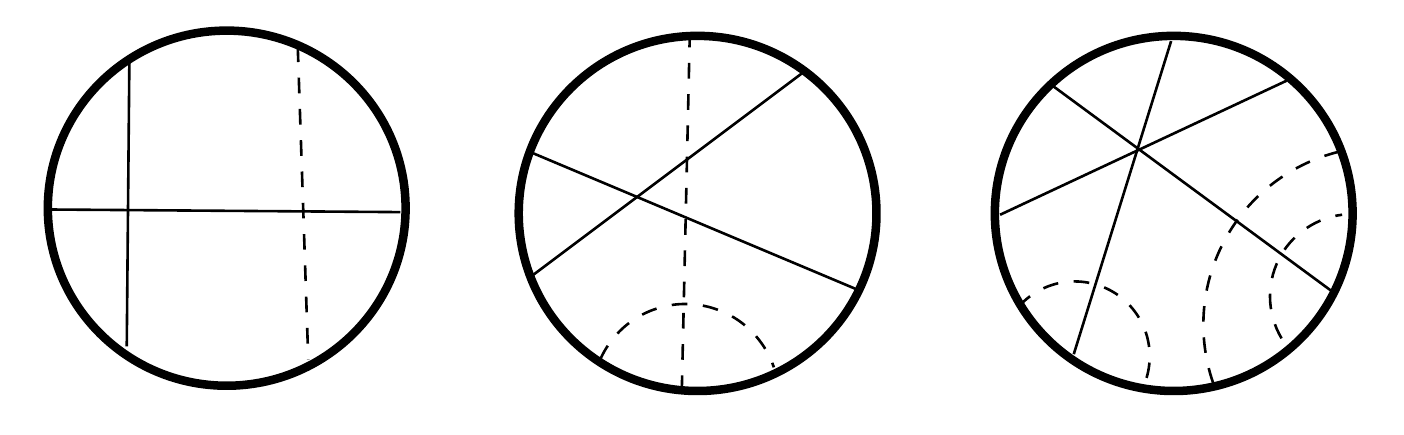}
\end{figure}

\section{Finite Type Invariants}

There is a striking similarity between the representation of knots though signed chord diagrams and the representation of knots through Gauss diagrams. There is one key difference however, which is of course that not all Gauss diagrams represent knots, but by construction, all signed chord diagrams represent knots. This difference allows one to obtain additional structure on the space of finite type invariants. 

Throughout this section, $\scr{D}$ will denote the set of all signed chord diagrams up to equivalence, $\scr{K}$ will denote the set of knot types, and $\text{ord}(D)$ denotes the order of a signed chord diagram $D$.

\begin{dfn}
A function $f: \scr{D}\to \R$ is said to be finite type of order $\leq n$ if for any pair of signed chord diagrams, $D'\subseteq D$, for which $\ord(D)-\ord(D') >n$, we have 
\begin{equation} \sum_{D'\subseteq H\subseteq D} (-1)^{\ord{D}-\ord{H}} f(H) = 0 \label{altsum}\end{equation}
\end{dfn}

\begin{thm}\label{invariants}
A function $V: \scr{K} \to \R$ is a finite type invariant of order $\leq n$ if and only if the function $V\Gamma: \scr{D}\to \R$ is finite type of order $\leq n$. 
\end{thm}

\begin{proof}
First assume that $V$ is a finite type invariant of order $\leq n$. We must prove that $V\Gamma$ is finite type of order $\leq n$. Suppose we have a pair of signed chord diagrams, $D'\subseteq D$, for which $\ord(D)-\ord(D') >n$. By replacing, in $\Gamma(D)$, all of the clasp surgeries for chords in $D\setminus D'$ with singular clasp surgeries, we obtain a singular knot with more than $n$ self intersections. The sum in Equation \ref{altsum} is then the alternating sum of our invariant over all resolutions of our singular knot, which must vanish because our invariant is finite type of order $\leq n$. 

Now, we assume that a knot invariant $V: \scr{K}\to \R$ is such that $V\Gamma$ is finite type of order $\leq n$. To see that $V$ is a finite type invariant of order $\leq n$, observe that Corollary \ref{singular} implies that any singular knot can be represented by a pair of signed chord diagrams $D'\subseteq D$, by taking $\Gamma(D)$, and replacing all of the clasp surgeries for chords in $D\setminus D'$ with singular clasp surgeries. This means any alternating sum of $V$ over resolutions for a singular knot with more than $n$ self intersections can be represented as the sum in Equation \ref{altsum}, which is zero by our assumption that $V\Gamma$ is finite type of order $\leq n$. This gives us our desired result. 
\end{proof}

It should be noted that Corollary \ref{singular} is slightly stronger than what is needed to prove Theorem \ref{invariants}. Indeed, given $s\in \{1,-1\}$, it suffices to check that $V\Gamma$ satisfies Equation \ref{altsum} only for diagram pairs $D'\subseteq D$ for which every chord in $D\setminus D'$ is of sign $s$. 

\begin{dfn}\label{notation}
We write $\scr{F}_n$ to denote the vector space of all functions $\scr{D}\to \R$ which are of finite type of order $n$, and we let $\scr{V}_n$ denote the subspace of $\scr{F}_n$ consisting of functions which factor through $\Gamma: \scr{D}\to \scr{K}$. We define $\scr{F} = \bigcup_{n = 1}^\infty \scr{F}_n$ and $\scr{V} = \bigcup_{n = 1}^\infty \scr{V}_n$. Theorem \ref{invariants} gives us a canonical isomorphism between $\scr{V}$ and the space of finite type invariants. We write $\scr{C}_n$ to denote the vector space of all functions $f: \scr{D}\to \R$ which are zero for all diagrams with order more than $n$. We define $\scr{C} = \bigcup_{n = 1}^\infty \scr{C}_n$. For $f: \scr{D}\to \R$, we define a function $C_f: \scr{D}\to \R$ by $$C_f(D) =  \sum_{H\subseteq D} (-1)^{\ord(G)-\ord(H)} f(H)$$
\end{dfn}

We will now begin to develop the connection between signed chord diagrams for knots, and the classical theory for finite type invariants in terms of chord diagrams. 

\begin{lem}\label{symbol}
Let $f\in V_n$ be a finite type invariant of order $n$. Let $D$ be a signed chord diagram of order $n$ with exactly $k$ strands of sign $-1$. Then the symbol of $f$ at the unsigned chord diagram underlying $D$ is equal to $(-1)^k C_f(D)$.
\end{lem}

\begin{proof}
If we replace the clasps of $\Gamma(D)$ with singular clasps, we get a singular knot whose chord diagram is the underlying chord diagram for $D$. Therefore, the symbol of the unsigned underlying chord diagram, which is the alternating sum of $f$ over all resolutions, is the alternating sum of $f$ over all subdiagrams of $D$, where including a negative clasp counts as a negative resolution, and including a positive clasp counts as a positive resolution. This gives us  $(-1)^k C_f(D)$ when one considers the sign change caused by all negative clasps.
\end{proof}

\begin{lem}
Let $f\in \scr{F}_n$. Then $C_f\in \scr{C}_n$.
\end{lem}
\begin{proof}
If $\ord(D) > n$ then Equation \ref{altsum} holds for the pair $\varnothing \subseteq D$, and the sum coincides with the sum which defines $C_f$. Therefore, $C_f(D) = 0$ when $\ord(D) > \ord(f)$. 
\end{proof}

\begin{lem}\label{inversion}
For any function $f: \scr{D}\to \R$, and any diagram $D$, we have $$ f(D) = \sum_{H\subseteq D} C_f(H) $$
\end{lem}
\begin{proof}
For every $D'\subseteq D$, the coefficient for $f(D')$ is an alternating sum over a cube of dimension $\ord(D) - \ord(D')$, which is zero unless $D = D'$. 
\end{proof}

\begin{dfn}
In a chord diagram, an isolated chord is a chord which does not cross any other chord when the chords are drawn as straight lines on a circular disk. 
\end{dfn}

\begin{lem}\label{moveone}
If a signed chord diagram $D$ has an isolated chord $y$, and $D'$ is $D$ with the isolated chord removed, then $\Gamma(D')$ is isotopic to $\Gamma(D)$. 
\end{lem}

\begin{proof}
Since the chord is isolated, we can take the band of $B(D)$ corresponding to the isolated chord, and move it through one revolution along the fibers of the unknot complement, tracing out an $S^2\times[0,1]$ embedded in $S^3$ which intersects $\Gamma(D)$ exactly at the clasp corresponding to the isolated chord, and which separates the parts of the knot on either side of the clasp. Rotating one side of the complement of this thickened sphere relative to the other by a full rotation will undo the clasp but leave the rest of the knot unchanged. 
\end{proof}

\begin{thm}\label{rel1}
Let $D$ be a signed chord diagram with an isolated chord. Let $V: \scr{K}\to \R$ be a knot invariant. We have $C_{V\Gamma}(D) = 0$.
\end{thm}

\begin{proof}
The sum for $C_{V\Gamma}(D)$ can be broken into two parts, the alternating sum over all subdiagrams which have the isolated chord, and the alternating sum over those that don't. \begin{align} (-1)^{\ord(D)}C_{V\Gamma}(D) =  \sum_{H\subseteq D} (-1)^{\ord{H}}V\Gamma(H) \nonumber\\
= \sum_{y\in H \subseteq D}(-1)^{\ord{H}}V\Gamma(H)+ \sum_{y\not\in H \subseteq D}(-1)^{\ord{H}}V\Gamma(H)\nonumber
\end{align}  Since removing the chord $y$ does not change the knot type, but it changes the order by one, we have
$$ (-1)^{\ord{D}}C_{V\Gamma}(D) =  -\sum_{y\not\in H' \subseteq D}(-1)^{\ord{H'}}V\Gamma(H')+ \sum_{y\not\in H \subseteq D}(-1)^{\ord{H}}V\Gamma(H) = 0$$ 
\end{proof}

\begin{lem}\label{pair}
Let $D_{\pm}$ be obtained from $D$ by inserting a pair of parallel, adjacent chords with opposite sign. Then $\Gamma(D)$ is isotopic to $\Gamma(D_{\pm})$
\end{lem}

\begin{proof}
We can cancel the clasps from the two chords because the tangle obtained by connecting a +1 clasp to a -1 clasp in parallel is trivial. 
\end{proof}

\begin{thm}\label{rel2}
Let $D_{\pm}$ be obtained from $D$ by inserting a pair of parallel, adjacent chords with opposite sign, $y_+$ and $y_-$. Let $D_+\subseteq D_{\pm}$ consist of $D$ with only the additional positive chord, and let $D_-\subseteq D_{\pm}$ consist of $D$ with only the additional negative chord. If $V: \scr{K}\to \R$ is a knot invariant, then $$C_{V\Gamma}(D_+) + C_{V\Gamma}(D_-) + C_{V\Gamma}(D_\pm) = 0$$
\end{thm}

\begin{proof}
The sum for $C_{V\Gamma}(D_{\pm})$ can be broken in to four parts, one part for each possible subset of the two parallel chords that remains in the subdiagram. \begin{align} (-1)^{\ord{D_\pm}}C_{V\Gamma}(D_\pm) =\nonumber\\
 \sum_{H\subseteq D}(-1)^{\ord{H}} V\Gamma(H) + \sum_{y_+\in H\subseteq D_+}(-1)^{\ord{H}} V\Gamma(H) +\nonumber\\
  \sum_{y_-\in H\subseteq D_-}(-1)^{\ord{H}} V\Gamma(H)+ \sum_{y_+,y_-\in H\subseteq D_{\pm}}(-1)^{\ord{H}} V\Gamma(H)\nonumber\end{align} By Lemma \ref{pair}, the part of the sum which has both chords from the pair is the same as the part of the sum with no chords from the pair. Thus, we have \begin{align} (-1)^{\ord{D_\pm}}C_{V\Gamma}(D_\pm) =\nonumber\\
 2\sum_{H\subseteq D}(-1)^{\ord{H}} V\Gamma(H) + \sum_{y_+\in H\subseteq D_+}(-1)^{\ord{H}} V\Gamma(H) +  \sum_{y_-\in H\subseteq D_-}(-1)^{\ord{H}} V\Gamma(H)\nonumber\\
 = \sum_{H\in D_+} (-1)^{\ord{H}} V\Gamma(H) + \sum_{H\in D_-} (-1)^{\ord{H}} V\Gamma(H)\nonumber\\
 = (-1)^{\ord(D_+)}C_{V\Gamma}(D_+) +  (-1)^{\ord(D_-)}C_{V\Gamma}(D_-)\nonumber\end{align}
 which simplifies to $$ C_{V\Gamma}(D_\pm) + C_{V\Gamma}(D_+) + C_{V\Gamma}(D_-) = 0$$

\end{proof}

\begin{dfn}\label{move}
Let $D$ be a signed chord diagram, and let $x$ be a positive chord of $D$. Select a preferred side of the chord $x$, and a preferred endpoint. We construct a diagram $D^x$ by removing $x$ and adding chords in the following manner. First, mark each chord which crossed $x$. Move along the preferred side of $x$ in the direction towards the preferred endpoint of $x$. Each time you reach an endpoint of an chord which crossed $x$, make two points, one just before the endpoint, and one just after the endpoint. We will label these points $p_1,...,p_{2n}$ in the order in which we have constructed them, where $n$ is the number of chords which cross $x$ in $D$. Then, at the preferred endpoint of $x$, we mark $2n$ adjacent points, $q_1,...,q_{2n}$, where labels increase in the direction toward the preferred side. Add chords between $q_i$ and $p_i$ of sign $(-1)^{i+1}$ for all $i = 1,...,2n$. (See the figure for an example.)
\end{dfn}

\begin{figure}[h]
\centering
\includegraphics[scale = 0.5]{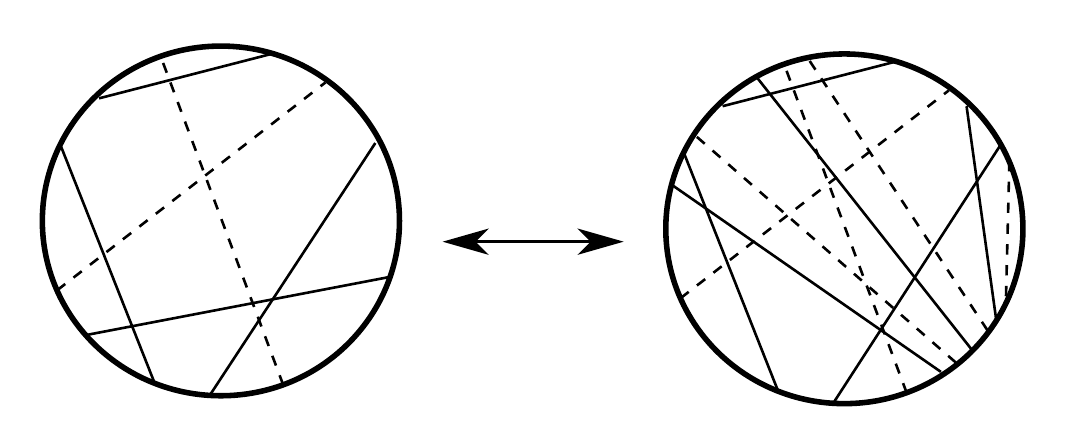}
\end{figure}

\begin{lem}\label{movethree}
Let $D$ and $D^x$ be related as in Definition \ref{move}. Then $\Gamma(D)$ is isotopic to $\Gamma(D^x)$.
\end{lem}

\begin{proof}
To describe the isotopy, we open up a clasp, loop a finger around the end of everything that crosses it, then pull everything down. See Figure \ref{hypermove}. 
\end{proof}

\begin{figure}[h]
\caption{\label{hypermove} The isotopy takes a positive clasp to an alternating sequence of clasps.}
\centering
\includegraphics[scale = 0.5]{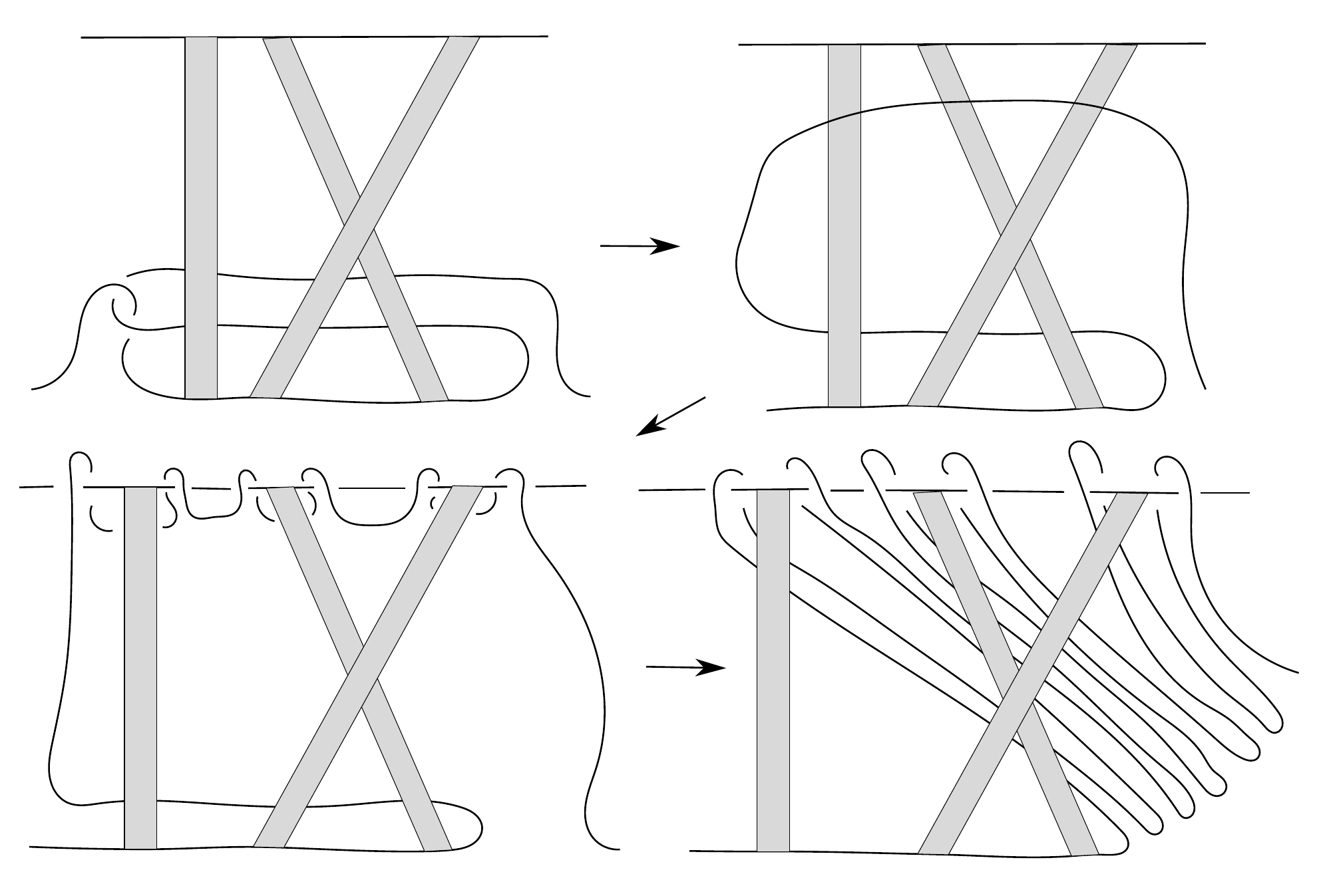}
\end{figure}

\begin{thm}\label{rel3}
Let $D$ be a signed chord diagram, and let $x$ be a chord in $D$. We mark four points on the boundary of the diagram, one to either side side of each endpoint of $x$, and we label them $p_1,p_2,p_3,p_4$, so that the pairs $(p_1,p_2)$ and $(p_3,p_4)$ are together at the endpoints, and the pairs $(p_1,p_3)$ and $(p_2,p_4)$ are each on some side of $x$. Now, select a point $q$ on the side of $x$ containing the points $p_2$ and $p_4$. We create diagrams $D_1,D_2,D_3,D_4$ in the following way. For $i \in \{1,2,3\}$, we add a positive chord $y$ between $q$ and $p_i$ to produce $D_i$, and $D_4$ is produced by adding a negative chord $y$ between $q$ and $p_4$. For $D_1$, we apply the procedure described in Definition \ref{move} to the chord $y$ with respect to the endpoint $q$ and the side of $y$ containing the points $p_3$ and $p_4$. This gives us a diagram $D_1^y$. We create a diagram $D_2^y$ similarly. Note that $D_2^y$ admits an inclusion into $D_1^y$. 

Let $V: \scr{K}\to \R$ be a knot invariant. The following equation holds true.

\begin{equation}C_{V\Gamma}(D_1) -C_{V\Gamma}(D_2) -C_{V\Gamma}(D_3) -C_{V\Gamma}(D_4) = \sum_{ H \subseteq D_1^y, H\not\subseteq D_2^y, \text{\emph{ord}} (H) \geq 2 + \text{\emph{ord}} (D)} C_{V\Gamma}(H)\label{fourterm}\end{equation}
\end{thm}

\begin{proof}
The sums for $C_{V\Gamma}(D_1)$ and $ C_{V\Gamma}(D_2)$ can each be broken into two parts, one for the subdiagrams which contain $y$ and one for those that don't. If $y$ is removed, the two diagrams become the same, so the terms coming from subdiagrams without $y$ cancel in $C_{V\Gamma}(D_1)- C_{V\Gamma}(D_2)$. We have \begin{align}
(-1)^{\ord{D_1}}(C_{V\Gamma}(D_1)- C_{V\Gamma}(D_2)) = \sum_{H\subseteq D_1}(-1)^{\ord{H}} V\Gamma(H) - \sum_{H\subseteq D_2}(-1)^{\ord{H}} V\Gamma(H) = \nonumber\\
\sum_{y\in H\subseteq D_1}(-1)^{\ord{H}} V\Gamma(H) + \sum_{y\not\in H\subseteq D_1}(-1)^{\ord{H}} V\Gamma(H) \nonumber\\ -\sum_{y\in H\subseteq D_2}(-1)^{\ord{H}} V\Gamma(H) - \sum_{y\not\in H\subseteq D_2}(-1)^{\ord{H}} V\Gamma(H)\nonumber\\
 = \sum_{y\in H\subseteq D_1}(-1)^{\ord{H}} V\Gamma(H) - \sum_{y\in H\subseteq D_2}(-1)^{\ord{H}} V\Gamma(H) \nonumber
\end{align} The remaining terms all have $y$, so we can apply the procedure from Definition \ref{move} and insert parallel chords as needed, which gives us an alternating sum over subdiagrams of $D_1^y$ and $D_2^y$ which contain all of the chords that were added in the procedure from Definition \ref{move}. 

\begin{align}
C_{V\Gamma}(D_1)- C_{V\Gamma}(D_2) = (-1)^{\ord{D_1}}\left(\sum_{y\in H\subseteq D_1}(-1)^{\ord{H}} V\Gamma(H) - \sum_{y\in H\subseteq D_2}(-1)^{\ord{H}} V\Gamma(H)\right) \nonumber\\
 =  \sum_{(D_1^y\setminus D)\subseteq H\subseteq D_1^y}(-1)^{\ord{D_1^y} -\ord{H}} V\Gamma(H) - \sum_{(D_2^y\setminus D)\subseteq  H\subseteq D_2^y}(-1)^{\ord{D_2^y} -\ord{H}} V\Gamma(H) \nonumber\\
 = \sum_{D \subseteq H \subseteq D_1^y} C_{V\Gamma}(H) - \sum_{D \subseteq H \subseteq D_2^y} C_{V\Gamma}(H) \nonumber
\end{align}

 The last equality holds because in each sum, the coefficient of $V\Gamma(H)$ for any subdiagram of $H\subseteq D_1^y$ is given by an alternating sum over a cube of dimension equal to the number of chords missing from $H$ that are also missing from $D$. This will cancel out to zero unless no chords are missing. We can now rearrange the equations by cancelling out what we can and combining all the terms of lowest order on the left side. This gives us the equation we wanted to prove $$ C_{V\Gamma}(D_1) -C_{V\Gamma}(D_2) -C_{V\Gamma}(D_3) -C_{V\Gamma}(D_4) = \sum_{ H \subseteq D_1^y, H\not\subseteq D_2^y, \ord (H) \geq 2 + \ord (D)} C_{V\Gamma}(H)$$
\end{proof}

 \newpage

\begin{thm}\label{gen}
Let $f\in \scr{F}$. Then $f\in \scr{V}$ if and only if the following conditions hold.
\begin{itemize}
\item[1)] If $D$ is a signed chord diagram with an isolated chord, then $$C_f(D) = 0$$
\item[2)] If $D_+,D_-$, and $D_\pm$ are as in Theorem \ref{rel2}, then $$ C_f(D_+) + C_f(D_-) + C_f(D_\pm) = 0 $$
\item[3)] If $D_1,D_2,D_3,D_4, D_1^y, D_2^y$ are as in Theorem \ref{rel3}, then $$ C_{f}(D_1) -C_{f}(D_2) -C_{f}(D_3) -C_{f}(D_4) = \sum_{ H \subseteq D_1^y, H\not\subseteq D_2^y, \text{\emph{ord}} (H) \geq 2 + \text{\emph{ord}} (D)} C_{f}(H) $$
\end{itemize}
\end{thm}

\begin{proof}
By Theorems \ref{rel1}, \ref{rel2}, and \ref{rel3}, $f$ will satisfy the three conditions if it is in $\scr{V}$. To prove the converse, let $\scr{X}_n$ be the subspace of $\scr{F}_n$ consisting of functions which satisfy conditions 1 through 3. We have an injective linear map $\scr{V}_n\to \scr{X}_n$ and this induces an injective map $\scr{V}_n/\scr{V}_{n-1} \to \scr{X}_n/\scr{X}_{n-1}$. Furthermore, we can prove that $\scr{X}_n/\scr{X}_{n-1}$ injectively maps into the space of weight functions on chord diagrams which satisfy the classical one and four term relations. To see why this is true, observe that the second equation allows us to represent vectors only by diagrams with positive chords, and the other two equations reduce to the classical one and four term relations. Finally, the fundamental theorem of finite type invariants \cite{IntroToVass} states that this space of weight functions is isomorphic to $\scr{V}_n/\scr{V}_{n-1}$. Therefore, our injective maps must be isomorphisms. By induction, this implies that our original linear map $\scr{V}_n\to \scr{X}_n$ was an isomorphism. 
\end{proof}

This theorem generalizes the fundamental theorem of finite type invariants by providing a kind of weight system which classifies all finite type invariants, rather than just the associated graded object. 

As a side note, one interesting application for signed chord diagrams is the construction of a relatively sparse family of knots such that any finite type invariant can be determined by its values on that family.

\begin{dfn}
We say a knot is totally positive (resp. totally negative) if it can be represented by a signed chord diagram with only positive (resp. negative) chords. 
\end{dfn}

These knots have some interesting properties.

\begin{thm}
Let $V$ and $V'$ be two finite type invariants. If $V$ and $V'$ agree on all totally positive knots (or all totally negative knots) then they agree on all knots. 
\end{thm}

\begin{proof}
By property 2 of Theorem \ref{gen}, we see that the value of a finite type invariant on a given diagram can be expressed as a linear combination of its values on totally positive, or totally negative diagrams. 
\end{proof}

\begin{thm}
For any knot $K$, there are only finitely many totally positive (or negative) diagrams $D$ with no isolated chords such that $\Gamma(D)$ is isotopic to $K$. 
\end{thm}

\begin{proof}
By Theorem \ref{gen}, the signed count of crossings in a signed chord diagram, where crossings between chords of the same sign are positive and crossings between chords of opposite sign are negative, is a finite type invariant of order 2. This means that the number of edges in the intersection graph obtained from a totally positive (or negative) diagram $D$ is a knot invariant. Furthermore, there are only finitely many graphs with a fixed number of edges and no isolated points, and there are only finitely many chord diagrams with the same intersection graph. This means there are only finitely many totally positive knot diagrams which have the same rank 2 finite type invariants as $K$. Of course, this also means there are only finitely many diagrams which represent $K$. 
\end{proof}

 \newpage
 
\section{The Proof of Theorem \ref{rep}}\label{main}

\begin{proof}
We describe a procedure to produce a signed chord diagram from a banded knot. We start with an orientable banded knot $X = (\gamma, \{\beta_i\}_{i\in F})$. In each step of the procedure, we will modify our representation for $X$. \\

\begin{figure}[h]
\caption{A Banded Knot.}
\centering
\includegraphics[scale = 0.56]{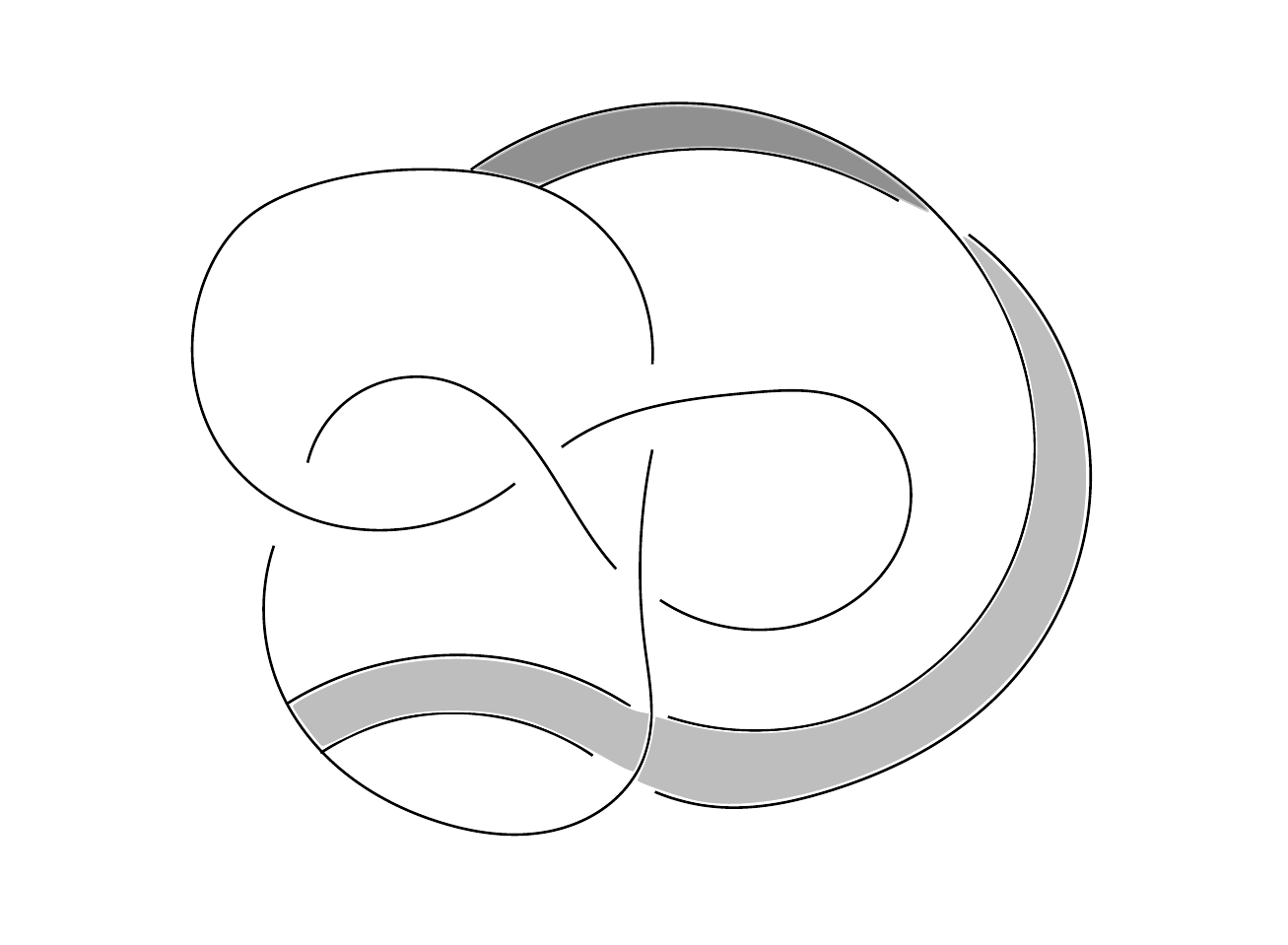}
\end{figure}

Step 1.

\item
 
Select a projection of $X$ onto $\R^2$, and isotope the bands so that they project into $\R^2$ in such a way that no edges cross over or under the bands, and the bands themselves are mapped via an orientation reversing embedding into $\R^2$.\\

\begin{figure}[h]
\caption{The above banded knot would look like this at this stage.}
\centering
\includegraphics[scale = 0.56]{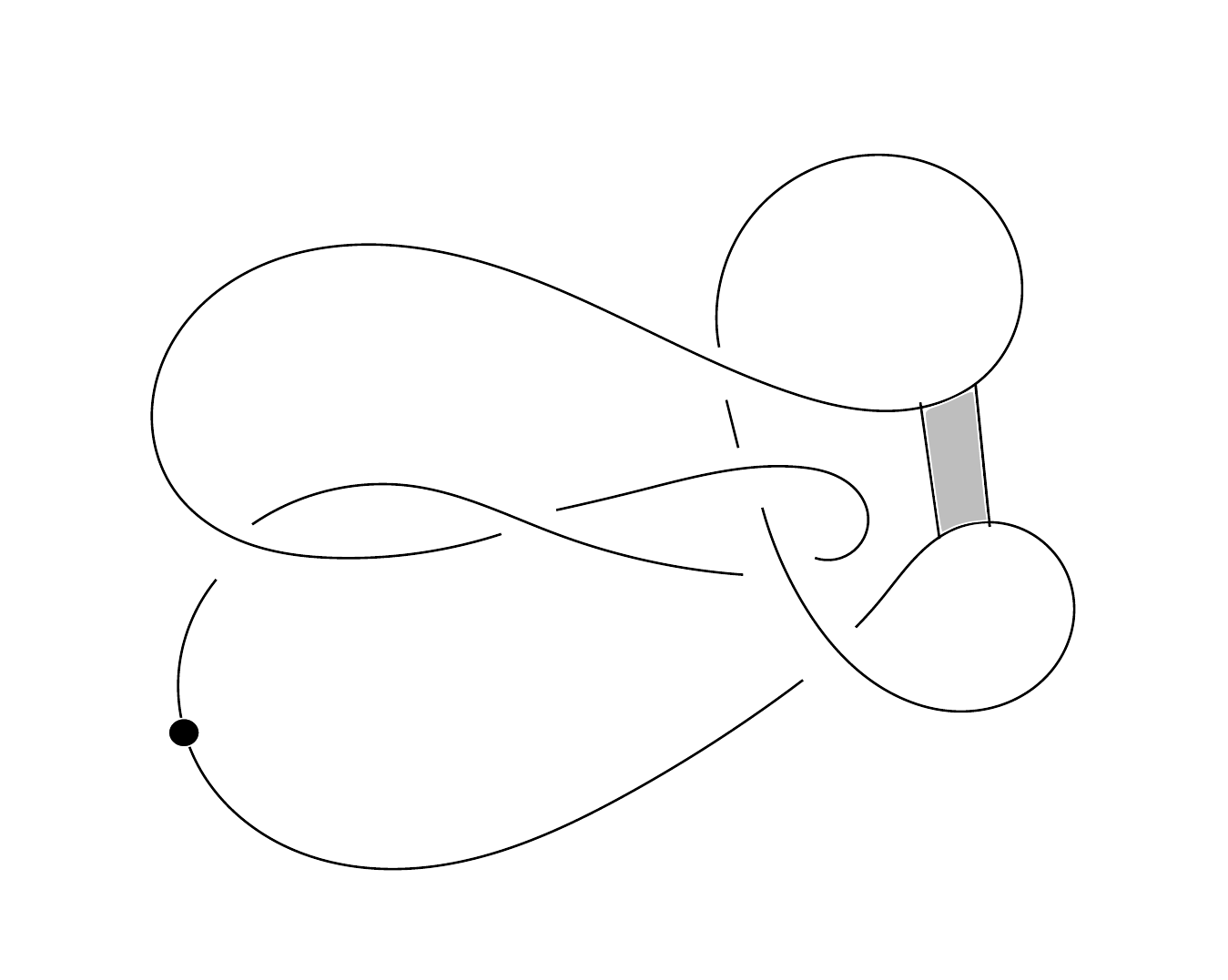}
\end{figure}

 We can think of the resulting projection as a knot diagram with some small squares connecting nearby edges of the diagram, representing bands.  \\

Step 2.

\item 

Select a base point for $X$ which does not lie on any crossing or band in the projected picture. Then, starting at the base point and moving in the positive direction along the knot, mark $S^1$ every time we reach an under-crossing for a crossing we have not yet reached, and every time we reach a given band for the first time. Then, isotope the knot so that an interval around every marked point and the base point projects onto the x-axis in $\R^2$ via an orientation preserving map, and the order in which these points appear along the x-axis is the same order in which the points appeared as we marked them, with the base point first. \\

\begin{figure}[h]
\caption{Our knot now looks like this.}
\centering
\includegraphics[scale = 0.45]{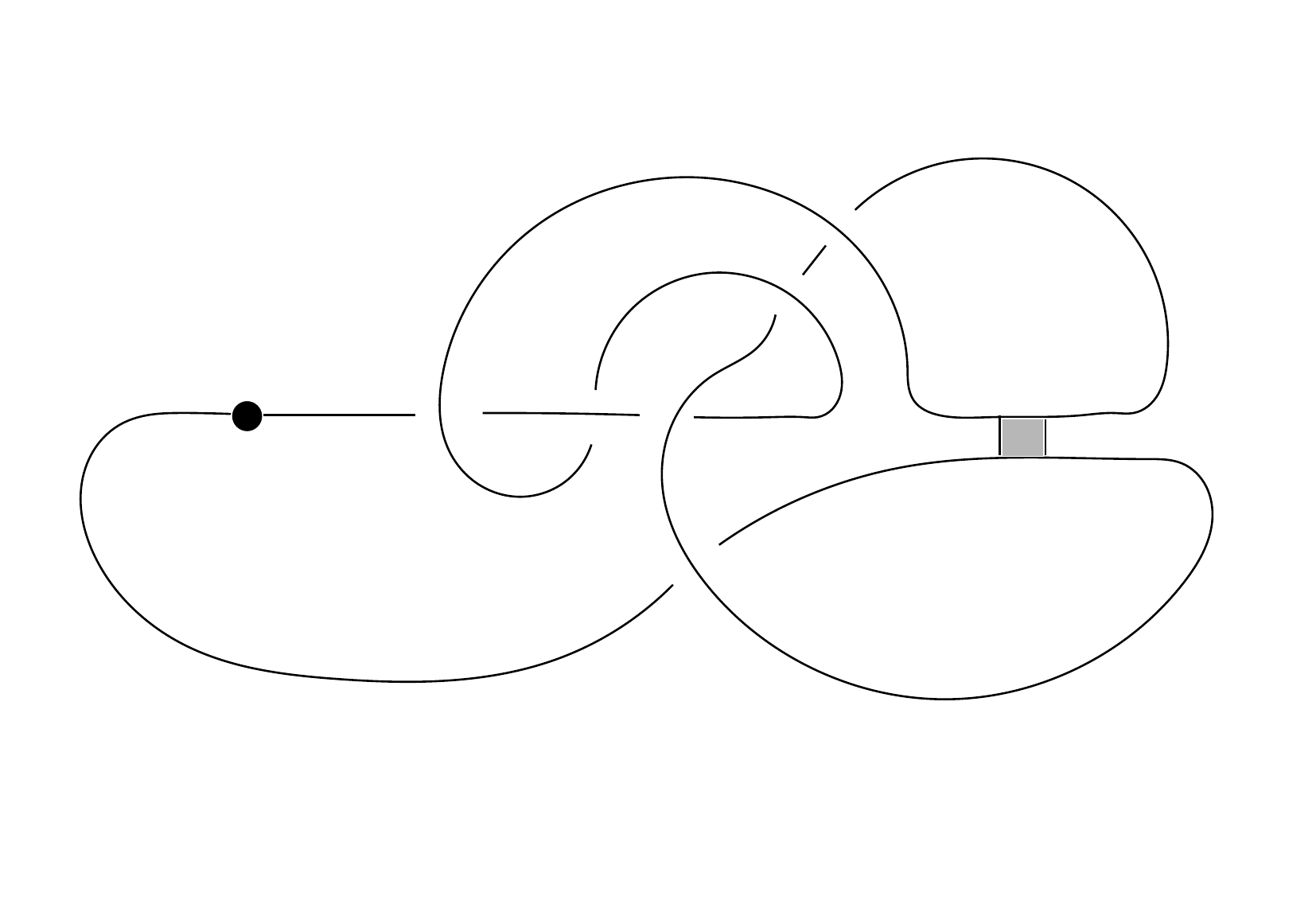}
\end{figure}

Step 3.

\item

We now view the knot from the side with an appropriate choice of height function. Each marked point now corresponds to a downward pointing finger, with either a clasp or a band at the end. Away from the fingers, moving in the positive direction along the knot corresponds to a decreasing y-coordinate. 

\begin{figure}[h]
\caption{What our knot looks like from the side.}
\centering
\includegraphics[scale = 0.5]{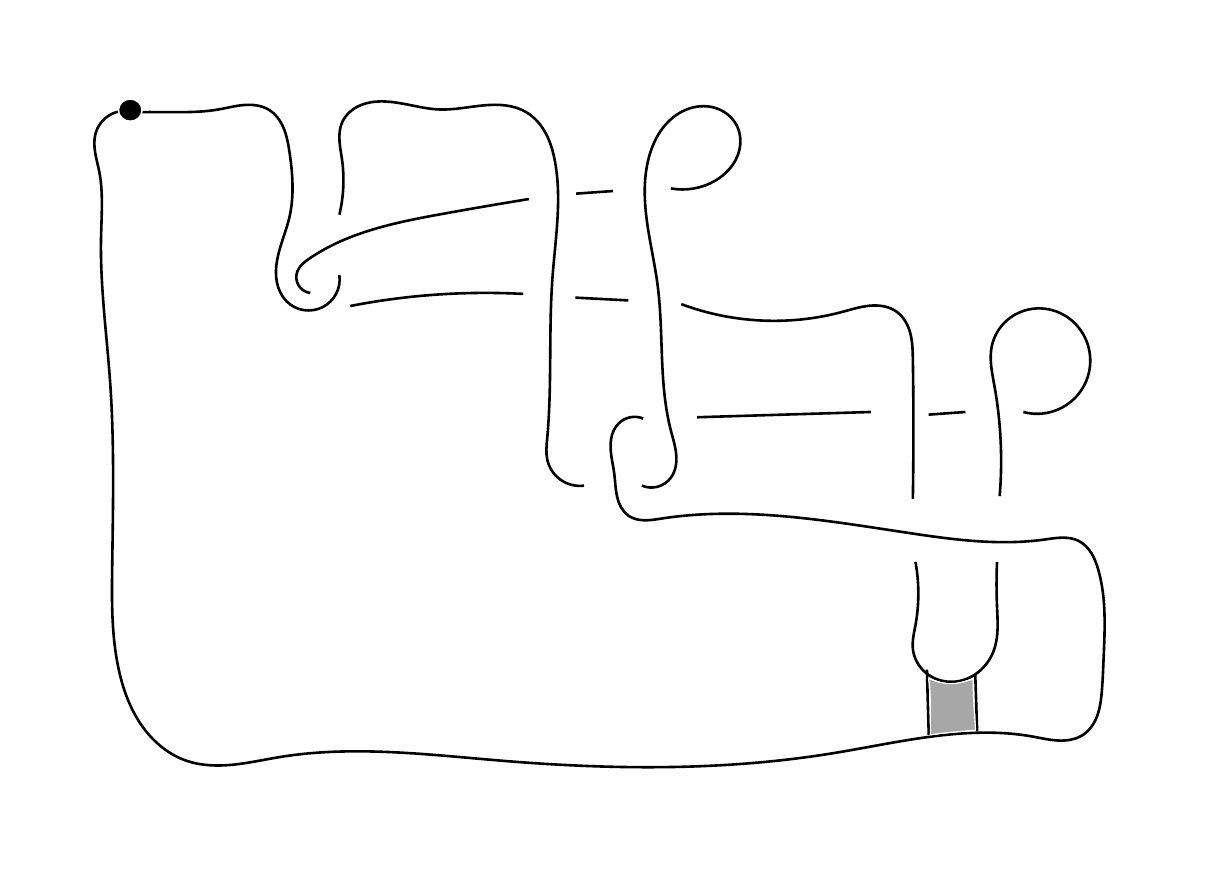}
\end{figure}

It is important to note that given any point $p$ on one of the fingers, if one draws a line down from that point, that line will only intersect the knot diagram at points that occur after $p$ relative to our base point. \\

Step 4. 

\item

Repeat steps 2 and 3, but this time we can ensure the marked points retain the order of their x-coordinates as we isotope them onto the x-axis. Due to the fact that after step 3, a downwards line from a finger could only intersect points that appeared later on the knot, we now have no lines which pass over any of the fingers, except lines which are involved in a clasp. 

\begin{figure}[h]
\caption{We repeat Step 2, lining up the marked points.}
\centering
\includegraphics[scale = 0.5]{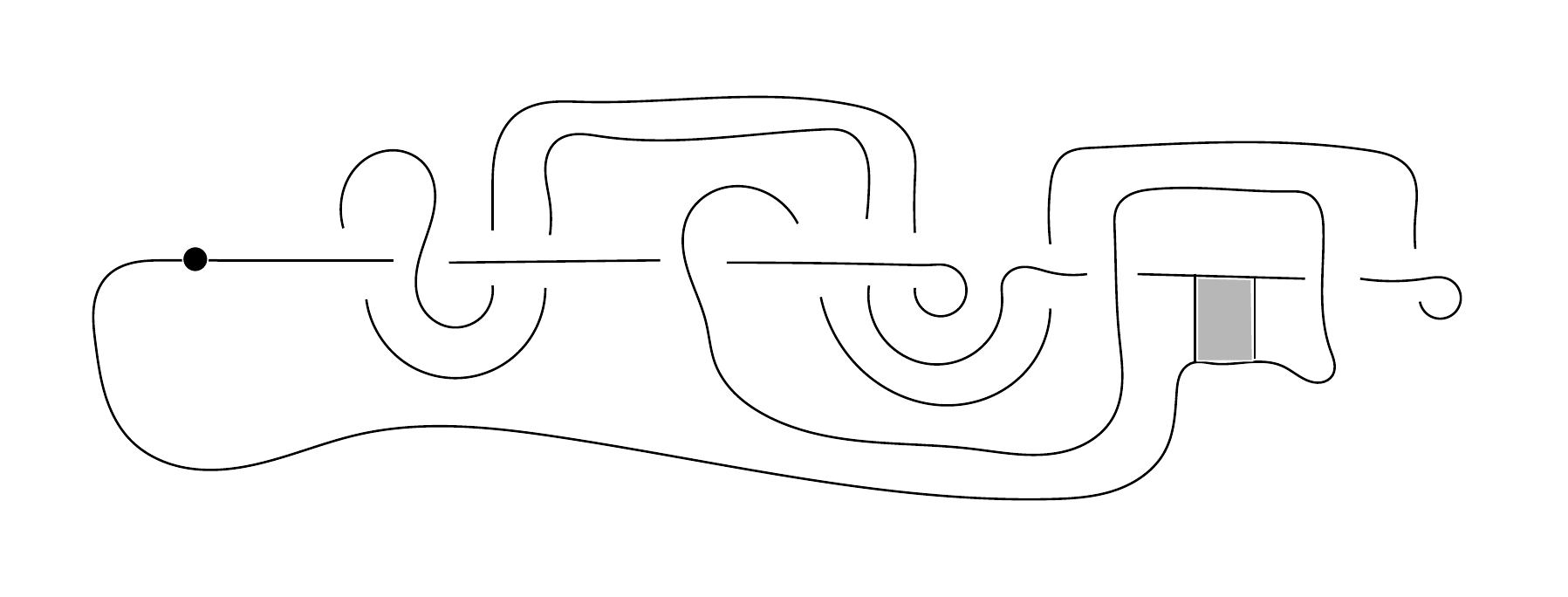}
\end{figure}

\begin{figure}[h]
\caption{We repeat Step 3. The downward fingers now lie on top of everything.}
\centering
\includegraphics[scale = 0.5]{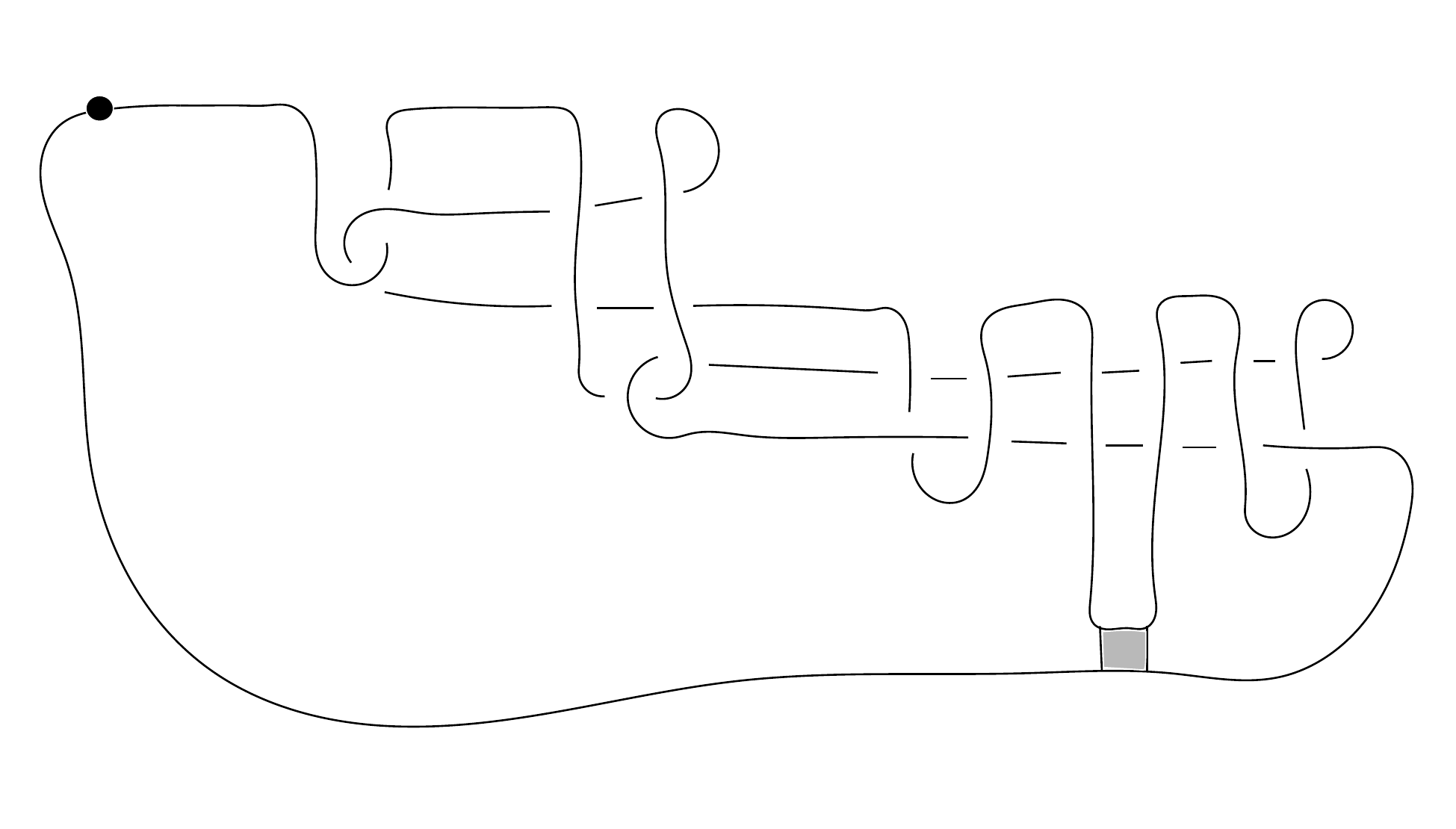}
\end{figure}

Now, any horizontal fingers which have many clasps should be separated into individual clasps. 

\begin{figure}[h]
\centering
\includegraphics[scale = 0.5]{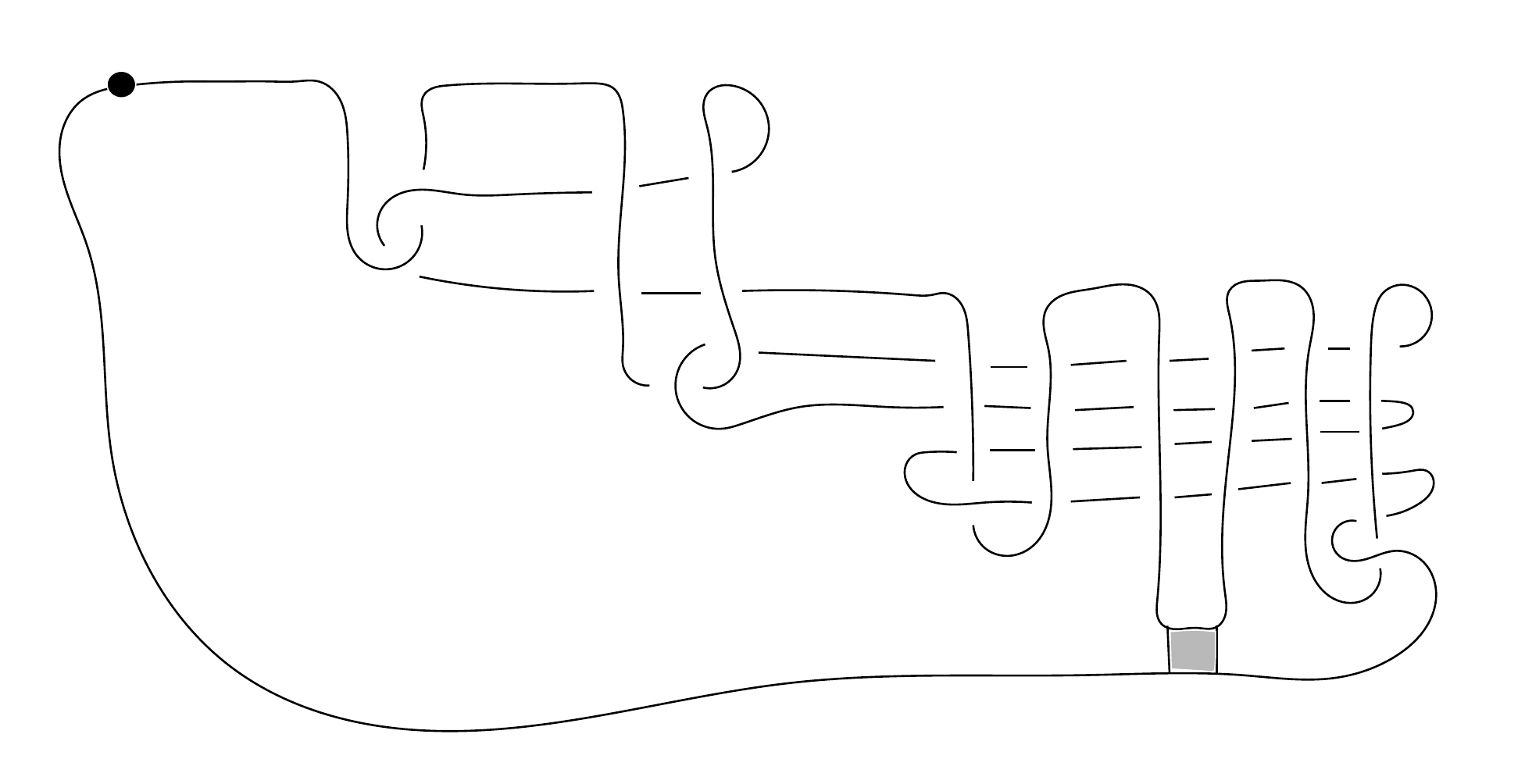}
\end{figure}

At this point, the clasps can be classified into four types: Bands, +1 clasps, -1 clasps, and -1 clasps with one full twist in the negative direction. This last kind of clasp can be transformed into a negative clasp with another small negative clasp at its base. 

\begin{figure}[h]
\caption{Our knot after removing the twisted negative clasp.}
\centering
\includegraphics[scale = 0.5]{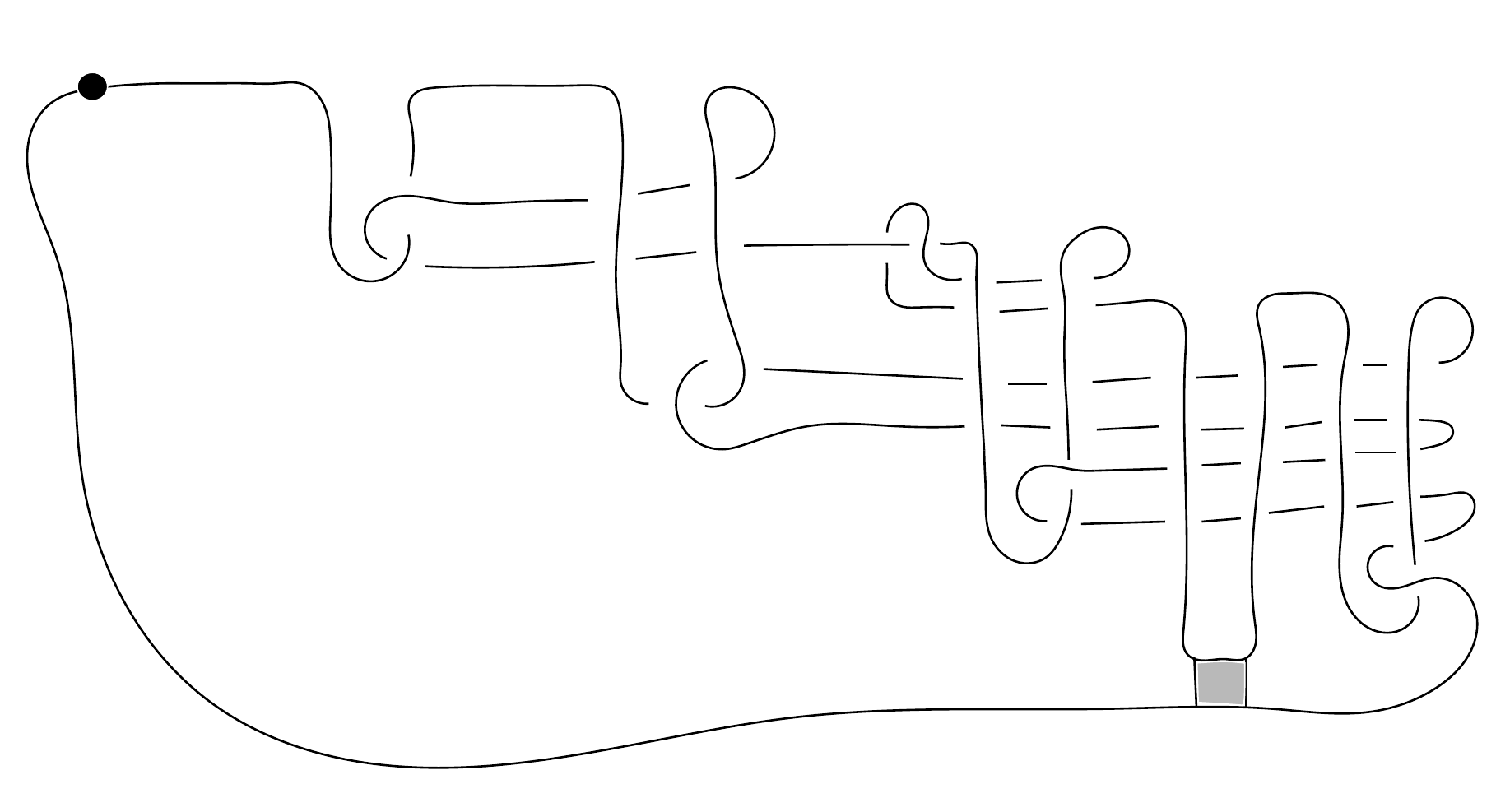}
\end{figure}

\begin{figure}[h]
\caption{Flip the knot from the previous figure upside down to get the final knot diagram.}
\centering
\includegraphics[scale = 0.4]{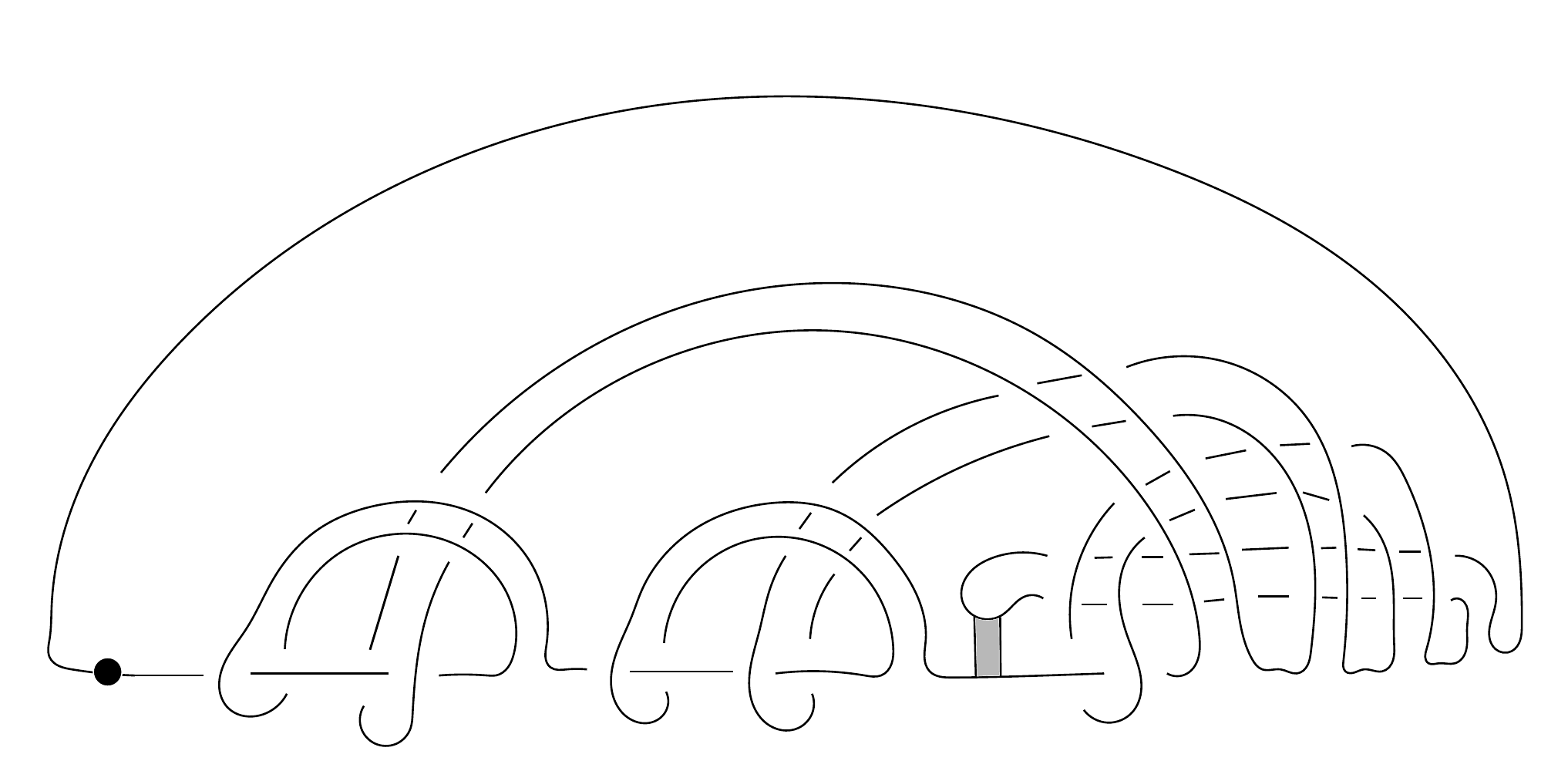}
\end{figure}

Now, we can produce a partially signed chord diagram by making a sign zero chord for each band, a sign +1 chord for each +1 clasp, and a sign -1 chord for each -1 clasp. 
\begin{figure}[h]
\caption{The final partially signed chord diagram for our example banded knot.}
\centering
\includegraphics[scale = 0.4]{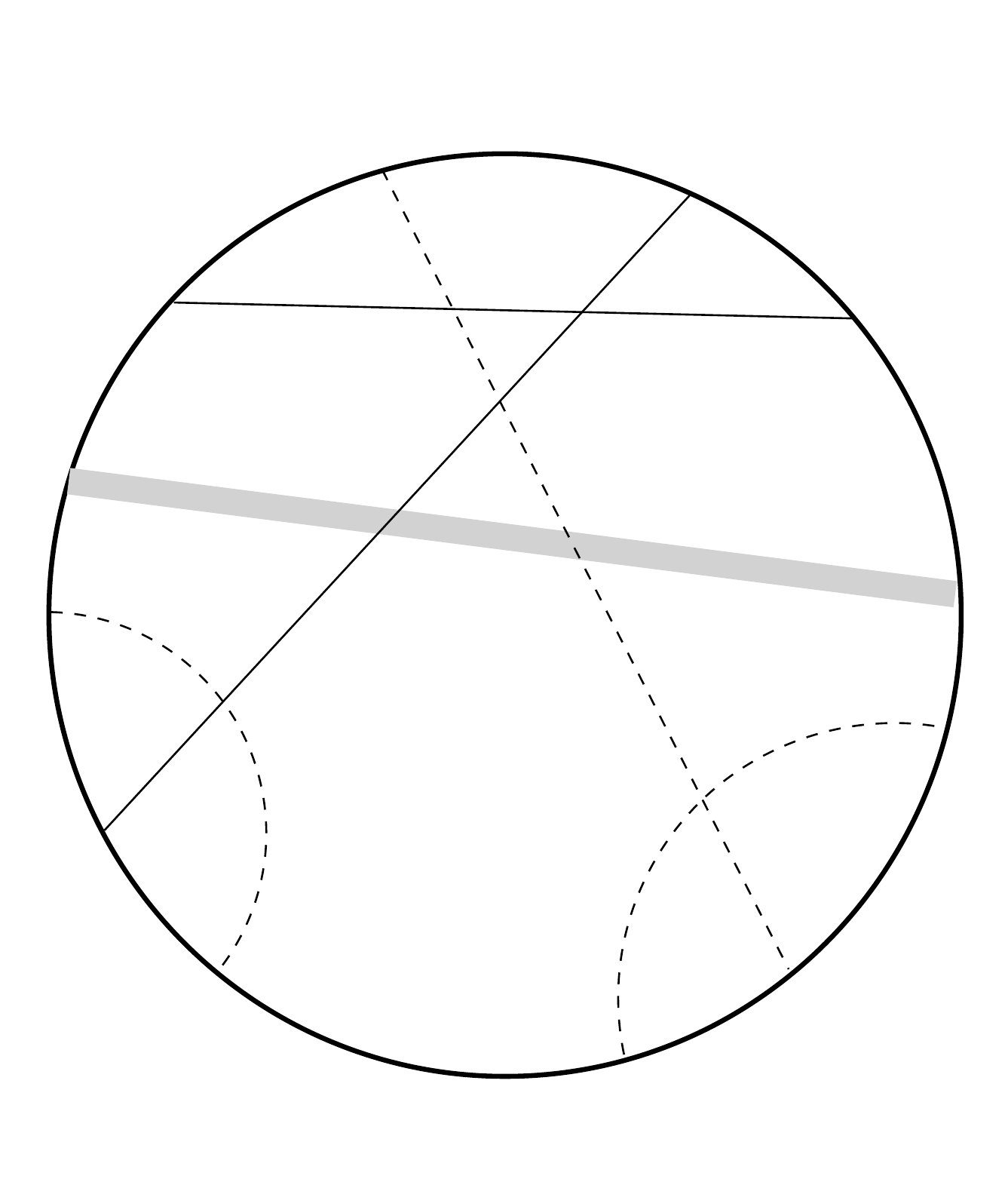}
\end{figure}

We have now produced a partially signed chord diagram $D$, and produced an isotopy from $K$ to $\Gamma(D)$.  (See figures below.)

\end{proof}

\newpage

\section{Equivalence Moves on Signed Chord Diagrams}

In this section, we will give a precise combinatorial description of the procedure described in the previous section, and use it to create a complete set of moves for signed chord diagrams. 

\begin{dfn}
A knot diagram consists of the following data. 
\begin{itemize}
\item[1)] A smooth immersion $p:S^1\to \R^2$ such that all self intersections are transverse double-points.
\item[2)] For each double point $x \in \R^2$, a bijection $h_x: p^{-1}(x) \to \{1,-1\}$ which we call the height function.
\item[3)] A choice of base-point $*\in S^1$ which does not map to a double-point.
\item[4)] A choice of orientation for $S^1$.
\end{itemize} 
\end{dfn}

\begin{dfn}
Suppose we have a knot diagram as above. Then let $X$ be the subset of $S^1$ consisting of points that map to double-points in $\R^2$. We have a function $h : X\to \{-1,1\}$ by combining the functions $h_x$ for each double point $x$. Furthermore, we can linearly order $X$ by starting at the base-point and moving in the positive direction in $S^1$. For $t\in X$, we call $t$ a gap of the diagram if it has $h(t) = -1$ and is ordered before the other point in $X$ which maps to $p(t)$ in $S_1$. Informally, if one draws a knot diagram starting at the base point, the gaps are the times when one needs to anticipate that a strand will later go over the line, so one must draw a gap in the line. 
\end{dfn}

\begin{dfn}
Given a knot diagram, let $t_1,t_2,...,t_n\in S^1$ be the gaps of the diagram in order starting from the base point. Let $t_0 = *$ be the base point. Then the knot diagram is said to be in gap linear position if $p(t_k)$ is on the x-axis for all $k$, with x-coordinate increasing as $k$ increases, and furthermore, there exist disjoint open intervals $I_0, I_1,...,I_n\subseteq S^1$ with $t_k\in I_k$ so that so that the function $p|_{I_k}: I_k\to \R^2$ has its image in the x-axis, and is increasing with respect to the x-coordinate as one moves in the positive direction in $S^1$. A gap linear knot diagram is a knot diagram in gap linear position. Two gap linear knot diagrams are considered equivalent if they are connected by a planar isotopy that preserves gap linearity throughout the isotopy.  
\end{dfn}

\begin{lem}\label{gaphomotopy}
Let $L$ be a gap linear knot diagram, with base point $t_0$ and gaps $t_1,...,t_n$. Then let $B_0,B_1,...,B_n$ be small disjoint open disks centered around $t_0,t_1,...,t_n$ respectively. These disks should be small enough that they each only contain one crossing and two arcs of the diagram so that the diagram crosses $B_i$ at exactly four points, except for $B_0$ which contains only one arc and has two intersections at its boundary. Let $S$ denote the set of all points at which the diagram intersects one of the disks. Let $I_1,I_2,...,I_m$ denote all of the closed intervals that comprise the complement of $p^{-1}\left(\bigcup_{k = 1}^n B_k \right)$ in $S^1$. For $k\in \{0,...,m\}$, let $A_k$ denote the set of all numbers in $\{0,...,n\}$ which are either $0$, or are the index of a disk for which some but not all of the points in $S$ on the boundary of that disk lie between the base point and a point in the interior of $I_k$.  Then the knot type represented by the diagram $L$ depends only on the homotopy type, relative boundary, of the path $p|_{I_k}: I_k\to \R^2\setminus\bigcup_{i\in A_k}B_i$ for all $k$.
\end{lem}

\begin{proof}
Outside of the disks there are no gaps, so when we are not in a disk we can think of the knot as having a decreasing height with respect to the linear order induced by starting at the base point and moving in the positive direction. From this, we see that outside the disks $B_i$ with $ i\in A_k$, there is a clear choice of whether a given point on the knot diagram should go under or over the strand defined by $p|_{I_k}$. Therefore, we see that if we apply a homotopy to $p|_{I_k}$, which doesn't cross any of the disks $B_i$ with $ i\in A_k$, there is a canonical choice for the height function at the crossings so that we get a gap linear knot diagram which represents the same knot as the original diagram. 
\end{proof}

\begin{dfn}
Fix $0 < \varepsilon < 1/2$. Let $n$ be a nonnegative integer. Let $$S=\{(-\varepsilon,0),(\varepsilon,0)\}\bigcup_{k = 1}^n \{(n-\varepsilon,0),(n+\varepsilon,0),(n,-\varepsilon),(n,\varepsilon)\}$$ For any point $a\in S$, let $a'\in S$ be the reflection of $a$ across the nearest integral lattice point. Let $B_k$ be the open ball of radius $\varepsilon$ around $(k,0)$.  A gap homotopy representation of order $n$ for a knot consists of the following data.
\begin{itemize}
\item[1)] A set pairs $p_k:\{0,1\}\to S $ defined for $k\in \{0,1,...,2n\}$ which form a partition for $S$, such that for all $k\in \{1,...,2n\}$, we have that $p_k(0)' = p_{k-1}(1)$ and the x-coordinate of $p_k(0)$ is greater than or equal to the x-coordinate of $p_{k-1}(1)$. Furthermore, we require $p_0(0) = (\varepsilon,0)$ and $p_{2n}(1) = (-\varepsilon,0)$. Finally, we require that the simple graph on $S$ with edges between all pairs of the form $(a,a')$ and $(p_k(0),p_k(1))$ is cyclic and connected, and the points in $S$ on the x-axis appear in order in this graph. We denote this graph by $G$.
\item[2)] For each $k\in \{0,1,...,2n\}$, a homotopy type of path $\overline{p}_k: [0,1]\to \R^2\setminus \bigcup_{i \in A_k}B_i$ that extends $p_k$, where $A_k$ denotes the set of all $k\in \{0,...,n\}$ for which $S\cap B_k$ fails to lie entirely in one connected component of $G$ after we remove the edges $((\varepsilon,0),(-\varepsilon,0))$ and $(p_k(0),p_k(1))$. 
\end{itemize}
By Lemma \ref{gaphomotopy}, a gap homotopy representation for a knot can be converted into a gap linear knot diagram by drawing crossings inside the disks $B_k$, (with the strands on the x axis going underneath) selecting generic representatives for $\overline{p_k}$, and then giving the crossings of the resulting knot diagram the unique height function for which the only gaps of the diagram are inside of the disks. Any gap linear knot diagram has a unique gap homotopy representation which it corresponds to in this way.
\end{dfn}

Fix the base point $* = (n+1,0)\in \R^2\setminus \bigcup_{k = 0}^n B_k$. Then if we fix a path in $\R^2\setminus \bigcup_{k = 0}^n B_k$ from each element of $S$ to $*$, we can represent homotopy types of paths by elements of a free group on $|A_k|$ generators. We select paths and generators as depicted in Figure \ref{paths}. The generator of the fundamental group which loops around $B_k$ will be denoted $x_k$. Now, $\overline{p}_k$ can be uniquely represented as a reduced word in the alphabet $\{x_i\}_{i\in A_k}$.

\begin{figure}[h]
\caption{\label{paths} Our generators for the fundamental group and our paths to the base point.}
\centering
\includegraphics[scale = 0.4]{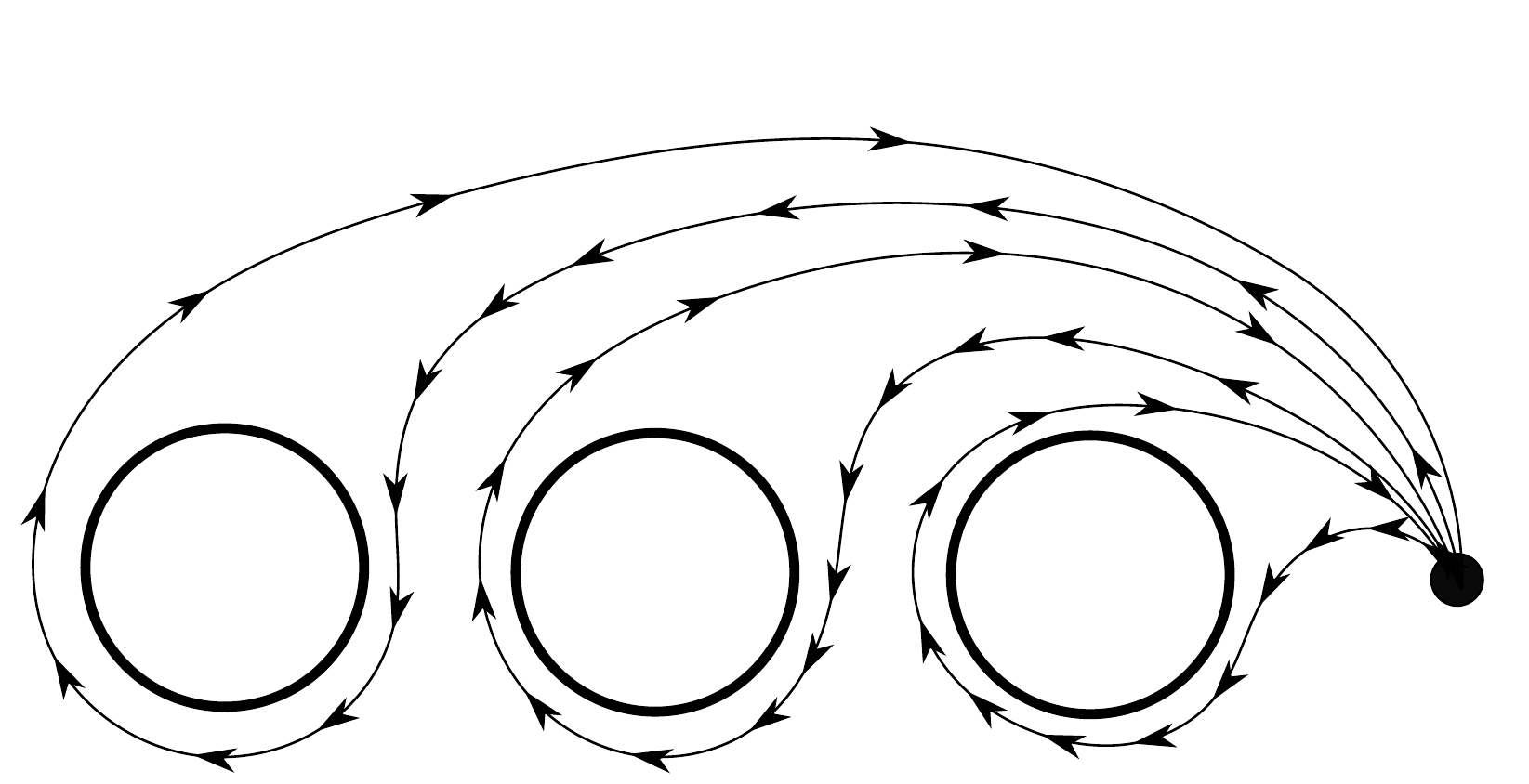}
\includegraphics[scale = 0.4]{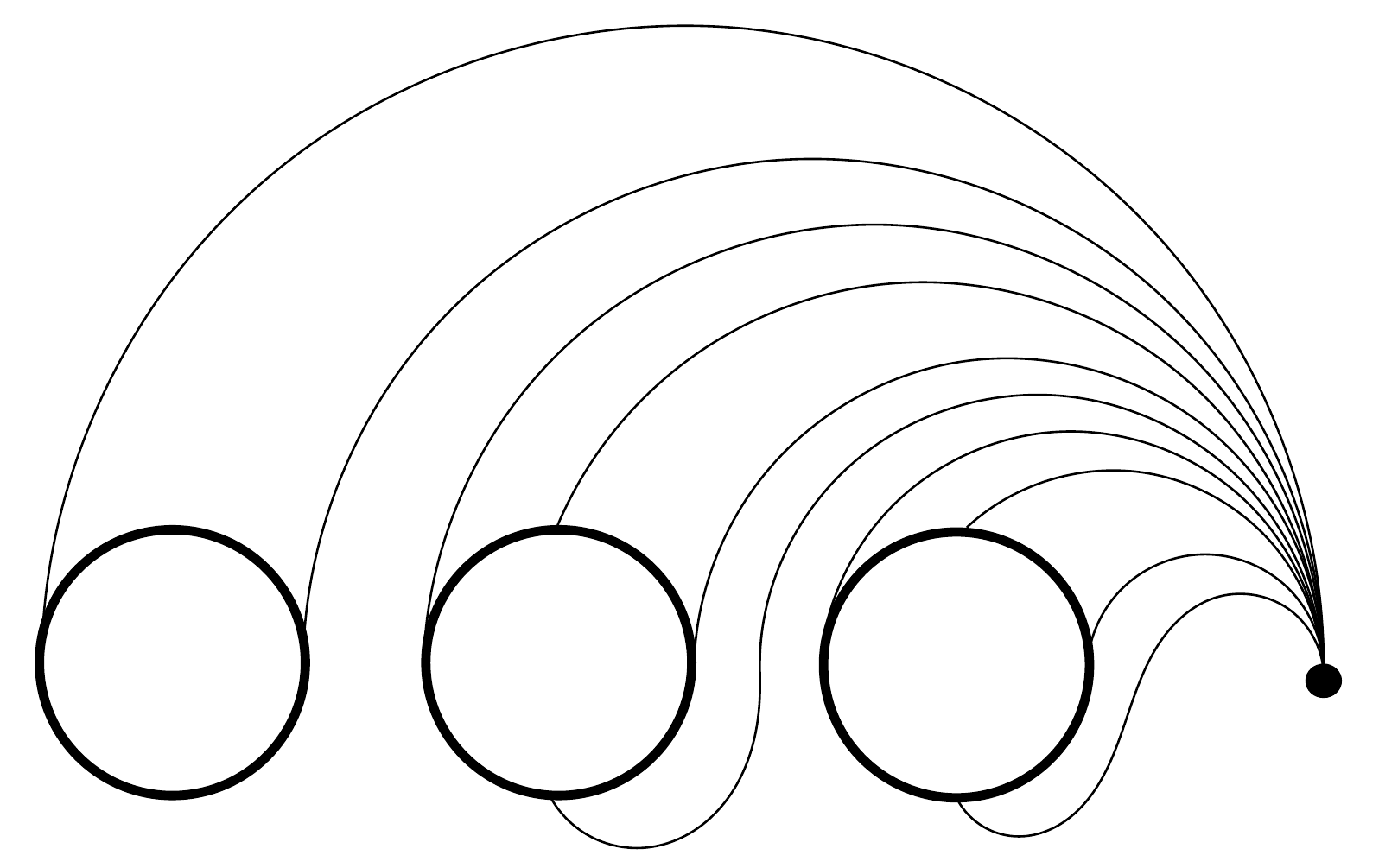}
\end{figure}

\begin{dfn}
A word sequence representation for a knot is a sequence of symbols in the set $\{x_n, x_n^{-1}, [_m, ]_m^{\pm}\}_{n\in \N, m\in \N\setminus\{0\}, \pm \in \{+,-\}}$. We require the following.
\begin{itemize}
\item[1)] The symbols $[_n$ and $]_n^{\pm}$ come in pairs, with $[_n$ appearing first. 
\item[2)] For no value of $n$ does $[_n$ appear more than once.
\item[3)] If $1 \leq m \leq n$ and $[_n$, $x_n$ or $x_n^{-1}$ appears, then so does $[_m$.
\item[4)] The symbols $[_1,[_2,...,[_n$ appear in order.
\end{itemize} 
\end{dfn}

A word sequence representation defines a knot because it gives us the necessary data to define a gap homotopy representation. The value of the largest $n$ for which $[_n$ appears is the order of the representation. At the beginning of the sequence, we start at $(\varepsilon,0)$. The words in $\{x_n\}_{n\in \N}$ describe a path in $\R^2\setminus \bigcup_{k = 0}^n B_k$. The symbol $[_n$ means passing from $(n-\varepsilon,0)$ to $(n+\varepsilon,0)$, the symbol $]_n^+$ means passing from $(n,\varepsilon)$ to $(n,-\varepsilon)$, and the symbol $]_n^-$ means passing from $(n,-\varepsilon)$ to $(n,\varepsilon)$. At the end of the sequence we finish at $(-\varepsilon,0)$. 

For example, $[_1x_1]_1^+$ defines a right handed trefoil, $[_1x_1^{-1}]_1^-$ defines a left handed trefoil, and $[_1x_1]_1^-$ defines a figure eight knot. It should be noted that many different word sequence representations may correspond to the same gap homotopy representation, since we do not require words to be reduced, nor that they be composed exclusively of generators with indices in $A_k$. 

At this point it should be noted that word sequence representations with no $x_i$ or $x_i^{-1}$ symbols correspond to signed chord diagrams. For instance, the signed chord diagram for the figure eight knot with a + and - chord that cross each other can be represented by $[_1[_2]_1^+]_2^-$. It is easy to see that the geometric realization of a signed chord diagram yields the same knot as the knot that is represented by the corresponding $x$-symbol free word sequence.

\begin{dfn}
Let $L_1$ and $L_2$ be two gap linear knot diagrams. Suppose there exists a knot diagram with the following properties. 
\begin{itemize}
\item[1)] We have two specified base points $*_1$ and $*_2$ so that as one moves in $S_1$ from $*_1$ to $*_2$ in the positive direction, only one crossing of the knot diagram is encountered.
\item[2)] Our diagram is gap linear with respect to both $*_1$ and $*_2$. With respect to $*_1$, our diagram is equivalent to $L_1$, and with respect to $*_2$, our diagram is equivalent to $L_2$. 
\item[3)] As one moves in $S^1$ from $*_1$ to $*_2$ in the positive direction, at no point in this path is the y-coordinate negative while the x-coordinate is the same as that of a base point or one of the gaps with respect to $*_1$ or $*_2$. That is to say, the path never crosses a line drawn straight down from a gap or base point. This path may, however, cross over a gap itself if the situation demands it. 
\end{itemize}
Then $L_1$ and $L_2$ are said to be connected by an elementary base point move.
\end{dfn}

\begin{dfn}
The pure framed braid group of $n+1$ strands acts on gap linear knot diagrams with $n$ gaps, where the action is the application of a planar isotopy that permutes the gaps and base point along the strands of a pure framed braid. If two gap linear knot diagrams are related by the action of an element of the pure framed braid group, then they are said to be related by a braid move.
\end{dfn}

\begin{lem}\label{lindiagmoves}
If $L_1$ and $L_2$ are two gap linear knot diagrams which represent the same knot, then they are related by a sequence of moves of the following types.
\begin{itemize}
\item[1)] Reidemeister moves that don't change the gap homotopy representation.
\item[2)] Elementary base point moves.
\item[3)] Braid moves.
\end{itemize}
\end{lem}

\begin{proof}
First, observe that given any Reidemeister move, if the base point is placed just before the top strand in the move, then applying the move does not change the gap homotopy representation. Furthermore, gap linear knot diagrams modulo braid moves are the same as knot diagrams, and any base point move for a knot diagram can be placed in elementary form if we can control the way the gaps are placed on the x axis through the braid group. Therefore, for any Reidemeister move we wish to apply, we can use braid moves and elementary base point moves to place the base point in a location where the Reidemeister move does not change the gap homotopy representation. This allows us to get between any two equivalent knot diagrams with just these three kinds of moves. 
\end{proof}

\begin{dfn}
Two gap homotopy representations $G_1$ and $G_2$ are said to be connected by a projected base point move if there exist gap linear knot diagrams $L_1$ and $L_2$ which are related by an elementary base point move, such that $G_1$ is the gap homotopy representation of $L_1$, and $G_2$ is the gap homotopy representation of $L_2$.
\end{dfn}

\begin{dfn}
The pure framed braid group acts on gap homotopy representations in the same way it acts on gap linear knot diagrams. We say two gap homotopy representations are related by a braid move if one is obtained from the other by an application of an element of the pure framed braid group. 
\end{dfn}

\begin{lem}\label{gaphommov}
If two gap homotopy representations represent the same knot, then one is obtained from the other by a sequence of projected base point moves and braid moves. 
\end{lem}

\begin{proof}
This is an easy corollary of Lemma \ref{lindiagmoves}. 
\end{proof}

\begin{lem}\label{basept}
If two gap homotopy representations $G_1$ and $G_2$ are connected by a projected base point move, then there are word sequences $w_1$ and $w_2$ representing $G_1$ and $G_2$ such that one of the following holds. 
\begin{itemize}
\item[1)] $w_2$ is obtained from $w_1$ by specifying a symbol $x_0^{\pm1}$ in $w_1$ and applying the following procedure. First, delete the specified $x_0^{\pm1}$ symbol. Then, increase every positive index by 1. Next, replace every $x_0$ with $x_1x_0$, and every $x_0^{-1}$ with $x_0^{-1}x_1^{-1}$. Finally, insert the symbol $]_1^{\mp}$ where the specified $x_0^{\pm1}$ symbol was, and put a $[_1$ at the very beginning. 

\item[2)] $w_2$ is obtained from $w_1$ by specifying a point in the sequence and applying the following procedure. First, let $k$ be the smallest positive integer for which $[_k$ does not appear before the specified point. For all $n \geq k$, increase every index $n$ by one. Then, insert $x_0[_k$ (resp. $[_kx_0^{-1}$) at the specified point, and append $]_k^+x_k$ (resp. $x_k^{-1}]_k^-$) at the end. Finally, for any symbol $s_i$ with index $i < k$, we replace it with the conjugation $x_ks_ix_k^{-1}$.
\end{itemize}
For example, $$[_1x_0x_1x_0]_1^+ \to [_1[_2]_1^- x_2x_1x_0]_2^+$$ or $$[_1x_1[_2x_2]_1^+]_2^+\to x_2[_1x_2^{-1}x_2x_1x_2^{-1}[_2x_2x_0^{-1}x_2^{-1}[_3x_3x_2]_1^+x_2^{-1}]_3^+x_2^{-1}]_2^-$$
\end{lem}

\begin{proof}
The first move comes from the case where the path between the two base points has a gap with respect to the first base point. We can think of this move as taking a loop around the base point, the specified $x_0^{\pm 1}$ symbol, and sliding it off of the base point to the right, producing a gap. Any other loop around the base point will now wrap around the base point and the first gap, so we replace $x_0$ with $x_1x_0$. 

The second move comes from the case where the path between the two base points intersects a gap for the second base point. We can think of this move as passing a strand from the upper half plane under the base point so that a new gap is formed. The gap must then be moved through the upper half plane and placed in the correct location on the x axis. The process of doing this conjugates other loops by the loop around the new gap. As we need to rotate the new gap into the correct orientation, we need to insert a $x_k^{\pm 1}$ symbol next to the closing bracket for the new gap. Also the process of passing a strand under the base point creates a loop around the base point which is the reason for adding $x_0^{\pm 1}$ next to the opening bracket for the new gap. 
\end{proof}

\begin{lem}\label{braidmv}
The pure framed braid group is generated by two types of generators: twisting a strand by 360 degrees, and moving one strand around the back of the braid, in front of another strand, and back to its original position. \cite{braids}  These generators act on word sequences as follows. For the first kind of generator, we can take every symbol that does not have $i$ as its index, and conjugate by $x_i^{-1}$. For the second kind of generator, given $i < j$, we can take every symbol that does not have $i$ or $j$ as its index, and conjugate by $x_i^{-1}x_j^{-1}$, unless we are between $[_i$ and $[_j$, in which case we conjugate by $x_j^{-1}x_i^{-1}$
\end{lem}

\begin{proof}
Take a disk closing a path that goes over every disk with a non specified index, and under every disk with a specified index. Applying a full twist along a thin neighborhood of the boundary of this disk gives the stated braid generators as well as the stated conjugations.
\end{proof}

\begin{lem}\label{reidmv}
If two word sequences represent the same gap homotopy representation, then one can be transformed to the other by applying the relations for words in a free group to the $x_n$ symbols, inserting or removing any amount of $x_n$ symbols outside of the brackets indexed by $n$, and inserting or removing any amount of $x_0$ symbols at the beginning of the word sequence.  
\end{lem}

\begin{proof}
This follows directly from the definition of a gap homotopy representation and the characterization of elements of a free group as reduced words. 
\end{proof}

\begin{lem}
If two word sequences represent the same knot, then one can be transformed into the other by a combination of the moves described in Lemmas \ref{basept}, \ref{braidmv}, and \ref{reidmv}.
\end{lem}

\begin{proof}
Follows from Lemma \ref{gaphommov}, and the fact that the moves in the lemmas describe the necessary moves on gap homotopy representaions. 
\end{proof}

\begin{lem}\label{sideview}
Suppose we have a word sequence. Then the following procedure does not change the knot type. First, select some value of $k$, and find the first time a symbol of the form $x_k$ or $x_k^{-1}$ appears inside of the brackets $[_k, ]_k^{\pm}$. Delete the symbol and record the place where it was. Then do the following. 
\begin{itemize}
\item[1)] If the symbol was $x_k$, then increment all indices of $k$ by one, and all indices greater than $k$ by two. Then, replace $[_{k+1}$ with $[_k[_{k+1}[_{k+2}$, and insert $]_{k+2}^-]_k^+$ where the $x_k$ symbol used to be.
\item[2)] If the symbol was $x_k^{-1}$, then increment all indices $k$ by two, and increment all indices greater than $k$ by three. Replace $[_{k+2}$ with $[_k[_{k+1}]_k^-[_{k+2}[_{k+3}$ and insert $]_{k+1}^-]_{k+3}^+$ where the $x_k^{-1}$ symbol used to be. 
\end{itemize}
\end{lem}
\begin{proof}
For the first case, we can take a loop around a gap, and flip it up over the gap to form two gaps. For the second case we can do a similar isotopy, but it adds an extra negative clasp to make everything lie in the correct order. (See Figure)
\end{proof}

\begin{figure}[h]
\centering
\includegraphics[scale = 0.7]{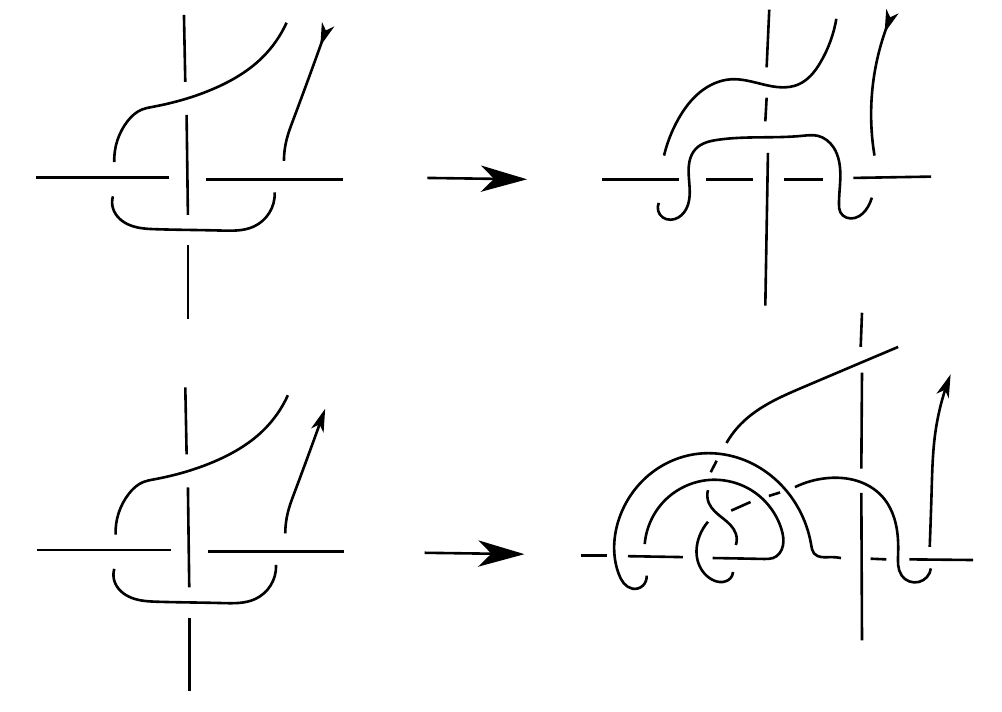}
\end{figure}

\begin{lem}
If we have a word sequence representation with $n$ as the largest index, then replacing an $x_0$ symbol with $x_1^{-1}...x_n^{-1}$ or replacing an $x_0^{-1}$ symbol with $x_n...x_1$ does not change the resulting knot type. 
\end{lem}

\begin{proof}
We can pass a loop around the base point over the point at infinity and we get a loop in the opposite direction which loops around everything else. 
\end{proof}

\begin{dfn}
Let $\scr{W}$ denote the set of all word sequence representations. Then we define a function $\Sigma: \scr{W}\to \scr{W}$ as follows. First, replace every $x_0$ with $x_1^{-1}...x_n^{-1}$ and every $x_0^{-1}$ with $x_n...x_1$. Then, delete any $x_n$ or $x_n^{-1}$ symbols that fail to lie within the brackets with their index. Finally, apply the procedure described in Lemma \ref{sideview} to every remaining $x$ symbol in order. 
\end{dfn}

\begin{lem}
The function $\Sigma$ preserves knot type, takes any word sequence to a word sequence with no $x$ symbols, and acts as the identity on word sequences with no $x$ symbols. 
\end{lem}

\begin{proof}
Every operation used preserves the knot type and does not affect $x$ symbol free sequences. Furthermore, after applying all of these operations, all of the $x$ symbols must be removed. 
\end{proof}

\begin{lem}
A word sequence with no $x$-symbols represents the same knot as the corresponding signed chord diagram. 
\end{lem}

\begin{proof}
Due to the fact that, away from the gaps, the height of our knot is always decreasing, we see that the clasps produced from the gaps lie in decreasing order with respect to the fibration of the unknot compliment. This means we have the geometric realization of the corresponding signed chord diagram. 
\end{proof}

\begin{dfn}
If $D$ is a signed chord diagram, an index block of $D$ is defined to be a subdiagram $H$ equivalent to $\Sigma([_1x_1^{s_1}x_1^{s_2}...x_1^{s_n}]_1^{\pm})$ where $s_i$ is either $+1$ or $-1$. The subdiagram $H$ must have the property that the interval around $[_1$ containing all of the left endpoints of the chords from the $x_i$ symbols contains no chords in $D\setminus H$. Furthermore, we require that the right endpoints of each $x$ symbol are adjacent in $D$. The inside of an index block refers to the interval between $[_1$ and $]_1^{\pm}$. The outside refers to the complement of the inside. 
\end{dfn}

\begin{lem}\label{indexrep}
Let $D$ have disjoint index blocks $H_1,H_2,...,H_n$ for which some point on the boundary of $D$ lies outside of all of the index blocks. Then there is a word sequence $w$ so that $\Sigma(w)$ is $D$, and the chords in $H_k$ are precicely the chords produced from the symbols in $w$ of some index $i_k$. 
\end{lem}

\begin{proof}
We construct $w$ as follows. First, make the base point of the diagram the point that lies outside all of the index blocks. Then, we can make bracket symbols for all chords in the diagram except those in the index blocks. For the index blocks, we make chords and $x$ symbols in the obvious way. We then see that the resulting word sequence corresponds to our original diagram. 
\end{proof}

\begin{dfn}
Let $D$ be a signed chord diagram and $H$ an index block of $D$. Let $p$ be an endpoint of a chord of $D$, which could be in $H$ or not. The operation of conjugating $p$ by $H$ is the following. We modify $D$ by adding chords to $H$ to make a larger index block $H'$ with an $x_1$ to the left of $p$ and an $x_1^{-1}$ to the right of $p$, unless either $x$ symbol would be outside of $H$, in which case that $x$ symbol is not included.  \end{dfn}

\begin{dfn}
Two signed chord diagrams $D_1$ and $D_2$ are said to be related by a braid move if there are two disjoint index blocks $H_1$ and $H_2$ in $D_1$ with nonempty intersection of their outsides, and $H_2$ placed to the right of $H_1$ with respect to a point outside of both index blocks, such that $D_2$ is obtained by conjugating every endpoint of every chord in $H_1\cup H_2$ by $H_2$, and then by $H_1$. 
\end{dfn}

Braid moves of signed chord diagrams are typically extremely complicated. 

\begin{figure}[h]
\caption{Example moves of type 1, 2, and 3 respectively.}
\centering
\includegraphics[scale = 0.8]{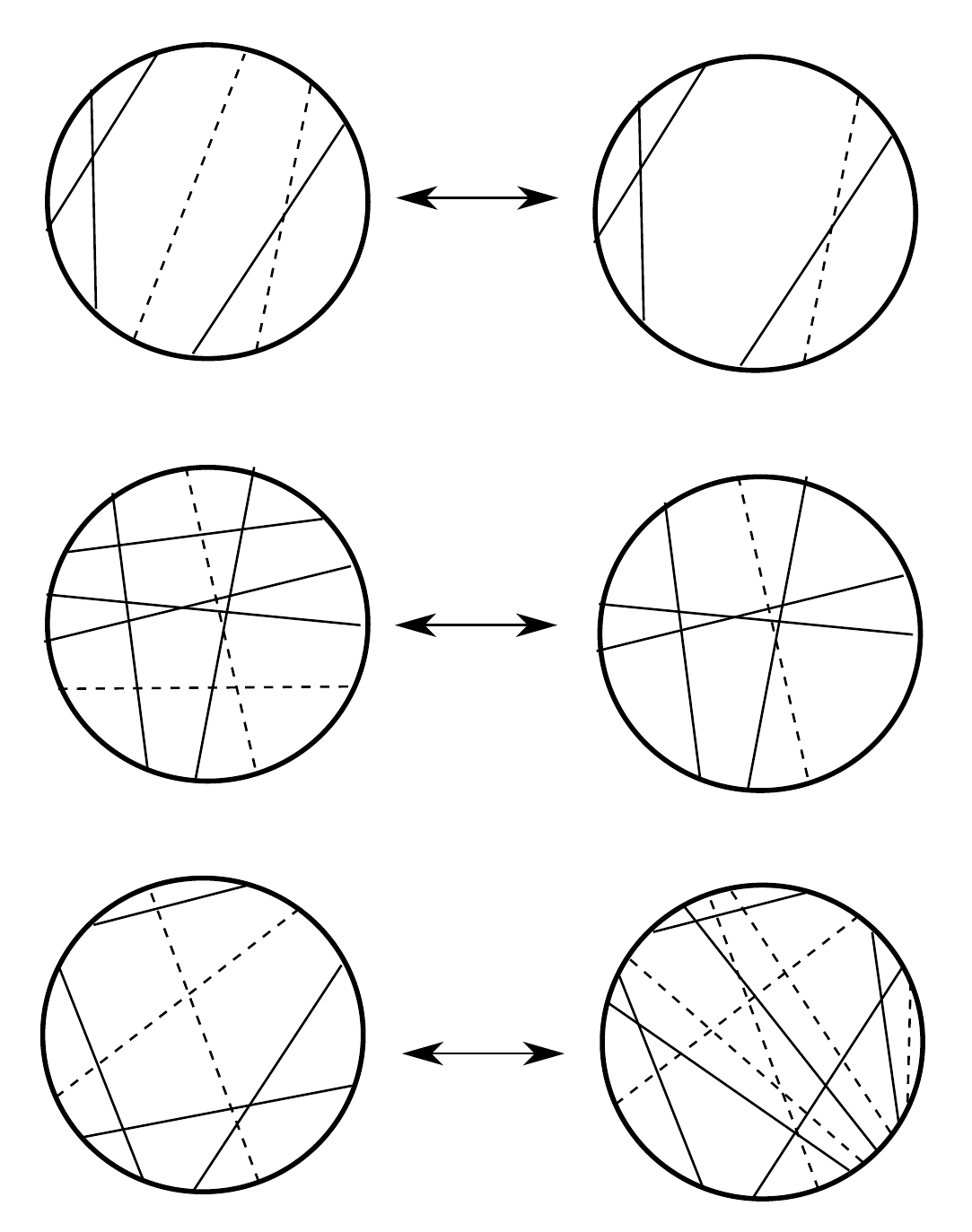}
\end{figure}

\begin{thm}\label{moves}
Two signed chord diagrams represent the same knot if and only if they are related by a sequence of moves of the following types:
\begin{itemize}
\item[1)] Insertion and deletion of isolated chords of either sign.
\item[2)] If two chords cross precisely the same set of chords, they have opposite sign, and they do not cross each other, then one can delete them both simultaneously. 
\item[3)] Remove a positive chord $x$ from the diagram after specifying a preferred side and a preferred endpoint of $x$. Then, mark each  chord which crossed $x$. Move along the preferred side of $x$ in the direction towards the preferred endpoint of $x$. Each time you reach an endpoint of a chord which crossed $x$, make two points, one just before the endpoint, and one just after the endpoint. We will label these points $p_1,...,p_{2n}$ in the order in which we have constructed them, where $n$ is the number of chords which crossed $x$. Then, at the preferred endpoint of $x$, we mark $2n$ adjacent points, $q_1,...,q_{2n}$, where labels increase in the direction toward the preferred side. Add chords between $q_i$ and $p_i$ of sign $(-1)^{i+1}$ for all $i = 1,...,2n$.
\item[4)] The braid moves described above.
\end{itemize}
\end{thm}

\begin{proof}
First, we must verify that the moves preserve the knot type. The first three moves are the same as those described in Lemmas \ref{moveone}, \ref{pair}, and \ref{movethree}, except that the second kind of move allows the two chords to be nonadjacent. To see that move 2 does not modify the knot type, we see that a clasp can be thought of as a full 360 degree twist in the strands it crosses. If two chords induce opposite twists in the same set of strands, those twists can be cancelled. 

To see why braid moves preserve the knot type, we apply Lemma \ref{indexrep} to see that whenever there is a braid move from diagram $D_1$ to diagram $D_2$, there is a word sequence representation $w_1$ with $\Sigma(w_1) = D_1$ such that applying a braid move to $w_1$ yields a gap homotopy representation which has a word sequence representative $w_2$ with $\Sigma(w_2) = D_2$.  One should observe that in the definition of braid moves for chord diagrams, we conjugate each $i$ or $j$ indexed symbol by $x_ix_j$, but in the definition for braid moves in Lemma \ref{braidmv}, we conjugate everything that is not an $i$ or $j$ symbol. These different conjugations yield the same gap homotopy representation because when $w_1$ and $w_2$ differ by a move from Lemma \ref{braidmv}, we can cancel all adjacent $x_kx_k^{-1}$ pairs that were added in the conjugation, and this has the effect of reversing the set of symbols which are conjugated. Now, we have that every index $i$ or $j$ symbol is conjugated by $x_ix_j$. 

Now we must show that these moves can get us between any two representations of the same knot. In order to do this, we will first describe some moves on signed chord diagrams that can be built out of moves of types 1, 2, and 3. 

First, insert two parallel strands between a pair of points on the boundary of the diagram. Now, at each endpoint of the negative chord, put small chords with oppsite sign. Now, apply a type 3 move to the small positive chord. This will yield one positive chord which crosses the negative chord, and a negative chord which can be cancelled with our first positive chord. What remains is three chords, a positive and negative chord which cross, and a small negative chord at one end of the negative chord. We now see that this structure can be inserted between any pair of points on the boundary of the diagram. We will call the insertion or deletion of such a structure a type $2'$ move. 

Now, we also describe a version of type 3 moves that apply to negative chords. We will first describe the procedure by which the move is defined, and then we will describe how it can be built out of moves of types 1,2, and 3. 

Remove a negative chord $x$ from the diagram after specifying a preferred side and a preferred endpoint of $x$. Then, mark each chord which crossed $x$. Move along the preferred side of $x$ in the direction towards the preferred endpoint of $x$. Each time you reach an endpoint of a chord which crossed $x$, make two points, one just before the endpoint, and one just after the endpoint. We will label these points $p_1,...,p_{2n}$ in the order in which we have constructed them, where $n$ is the number of chords which crossed $x$. Then, at the preferred endpoint of $x$, we mark $2n$ adjacent points, $q_1,...,q_{2n}$, where labels increase in the direction toward the preferred side. Add chords between $q_i$ and $p_{2n+1-i}$ of sign $(-1)^{i+1}$ for all $i = 1,...,2n$. Finally, at one endpoint of each of the negative chords we just added, insert a small negative chord. We will call this move a type $3'$ move. Note that it does not matter which endpoint of the negative chords have the small negative chord because the small chord can be moved to either endpoint with type 1 and 2 moves.

\begin{figure}[h]
\caption{Example moves of type $2'$ and $3'$ respectively.}
\centering
\includegraphics[scale = 0.8]{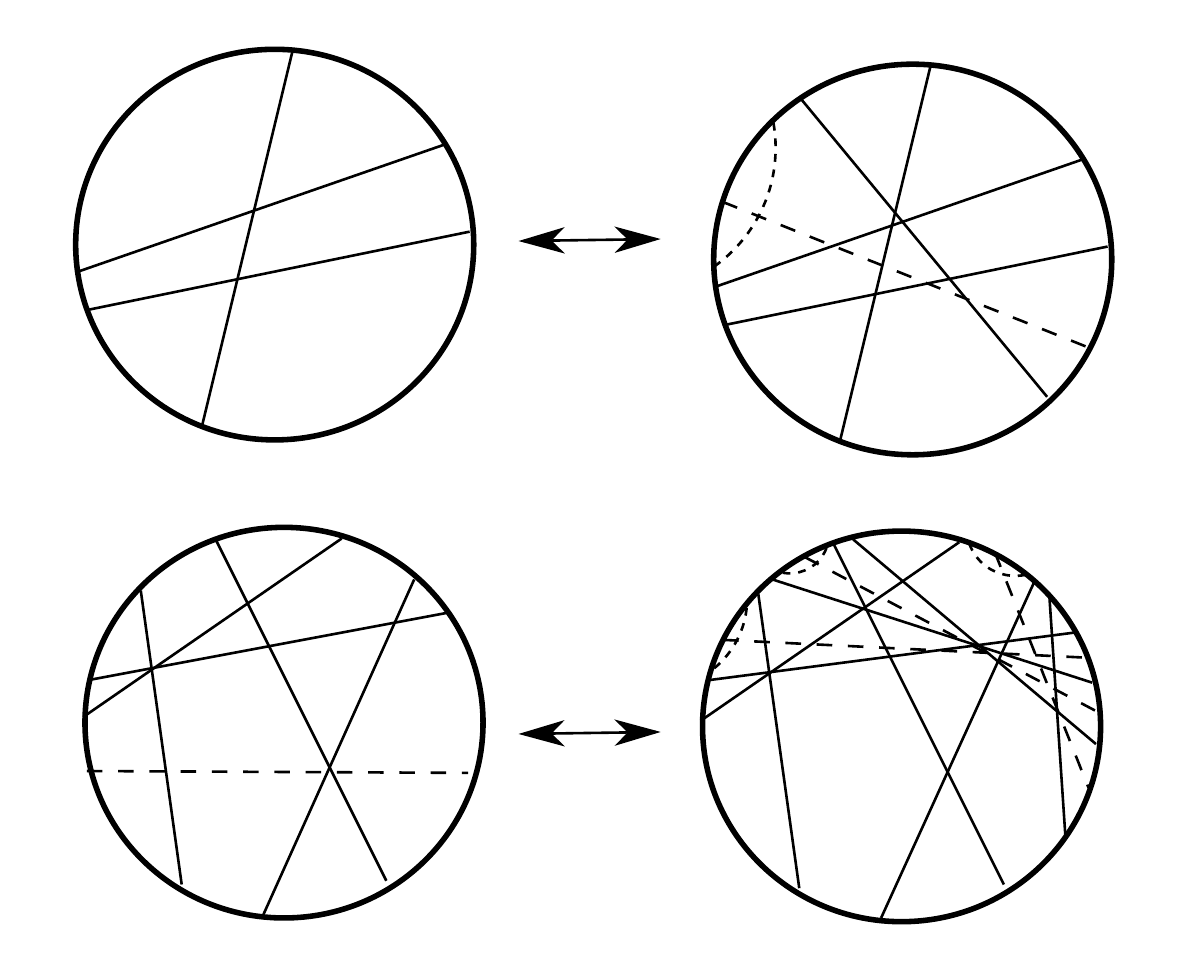}
\end{figure}

To see why type $3'$ moves come from moves of type 1, 2, and 3, First, take a diagram after a type $3'$ move has been applied to a negative chord. Insert a pair of parallel positive and negative chords where the negative chord used to be, and positioned so that the parallel chords do not cross any of the chords that were added in the type $3'$ move, and the positive one is closer to them. We now apply a type 3 move to the positive chord we just added with the same specified side and endpoint as the type $3'$ move. Now, we can apply a sequence of type 2 and type $2'$ moves to delete everything except the negative chord we added, thereby undoing the type $3'$ move. 

Let $w_1$ and $w_2$ be two word sequences which are related by a move from one of the lemmas \ref{basept}, \ref{braidmv}, and \ref{reidmv}. We wish to show that $\Sigma(w_1)$ and $\Sigma(w_2)$ are related by a sequence of moves of types 1,2,3, and 4. We will systematically go through every possible move to show that they can be built out of the moves we have defined.

First, we consider the moves from Lemma \ref{reidmv}. These moves consist of either deleting an $x$ symbol that lies in the front of the sequence or not inside of the brackets for its index, or of applying free group moves to the words. Deleting $x$ symbols as described has no effect on $\Sigma$ because we do this in the definition of $\Sigma$. If we insert or delete a pair of cancelling letters, then this corresponds to doing something more nontrivial. One symbol will have the effect of inserting two parallel chords of opposite sign which straddle the chord coming from the gap of its index, and the other will have a positive and negative chord which cross each other, straddle the chord for their index, and has a small negative chord at the end of the main negative chord. We can then delete all of these chords with one type 2 move and one type $2'$ move.  

\begin{figure}[h]
\caption{A move from Lemma \ref{reidmv}. $\Sigma([_1]_1^+)$ and $\Sigma([_1 x_1x_1^{-1}]_1^+)$}
\centering
\includegraphics[scale = 0.7]{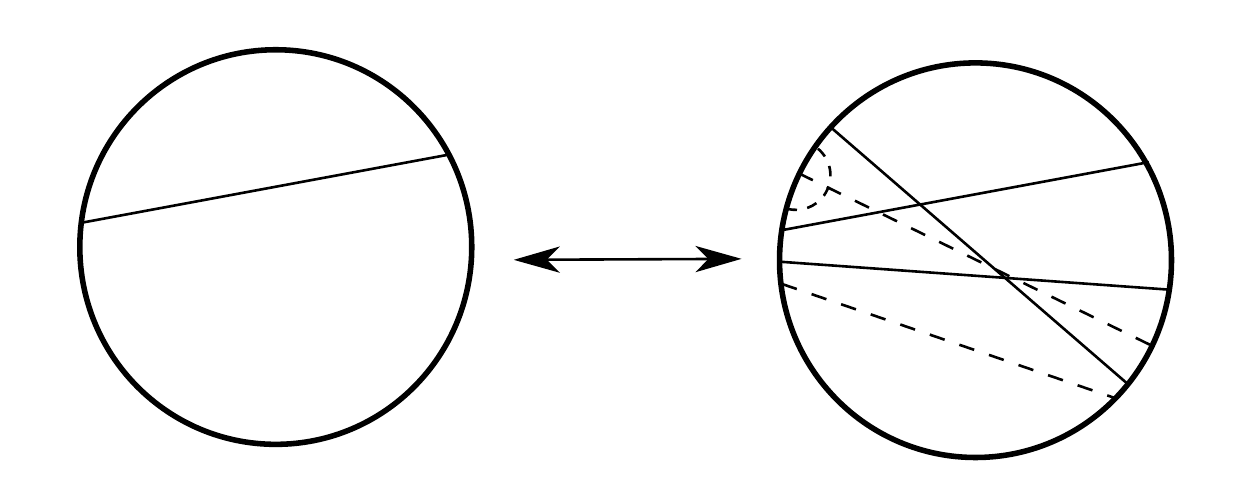}
\end{figure}

Now we consider the moves from Lemma \ref{basept}.  If $w_1$ and $w_2$ differ by the first kind of move, then $\Sigma (w_1)$ and $\Sigma (w_2)$ differ as signed chord diagrams by a single type 3 or type $3'$ move. The chords for the move come from the sequence of symbols produced from the $x_0$ symbol when we apply $\Sigma$. 

\begin{figure}[h]
\caption{A Lemma \ref{basept} move (first kind).  $\Sigma([_1[_2x_0]_1^+]_2^+)$ and $\Sigma([_1[_2[_3]_1^-]_2^+]_3^+)$}
\centering
\includegraphics[scale = 0.56]{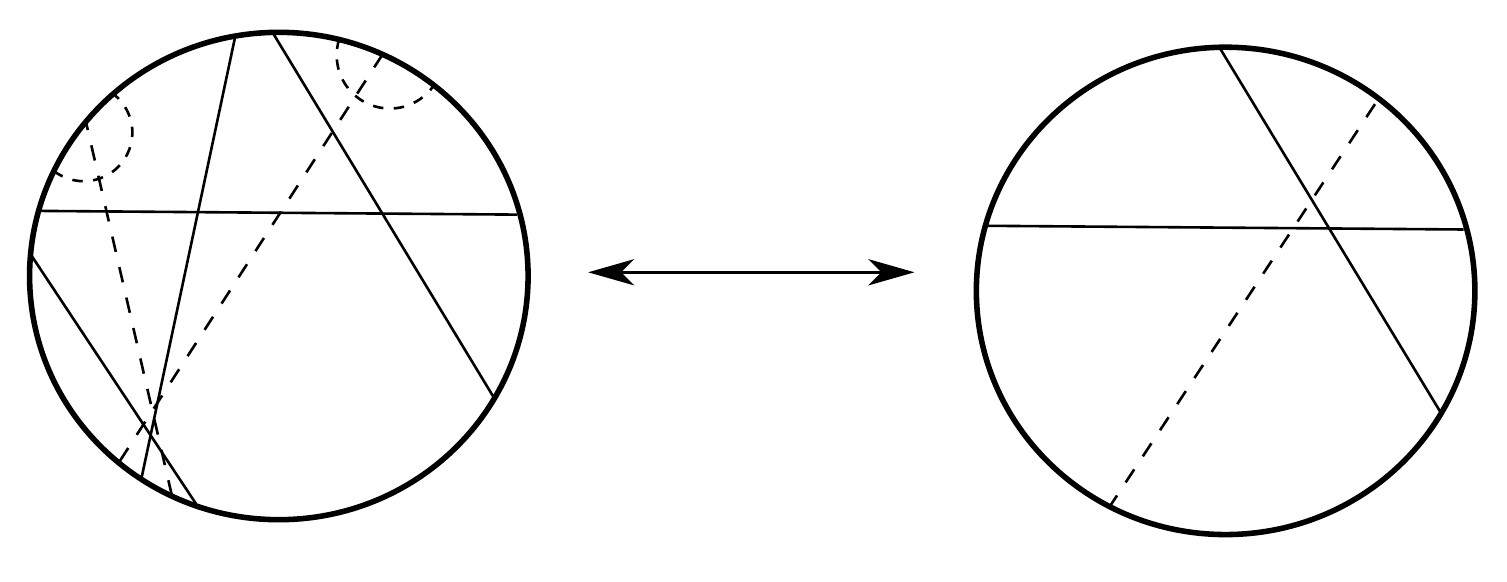}
\end{figure}

The second kind of move from Lemma \ref{basept} and the move from Lemma \ref{braidmv} both involve repeatedly conjugating certain symbols. The effect this has on $\Sigma$ is to add many alternating chords which group together to form potential type 3 and type $3'$ moves. 

\begin{figure}[h]
\caption{The effect of repeated conjugation.}
\centering
\includegraphics[scale = 0.6]{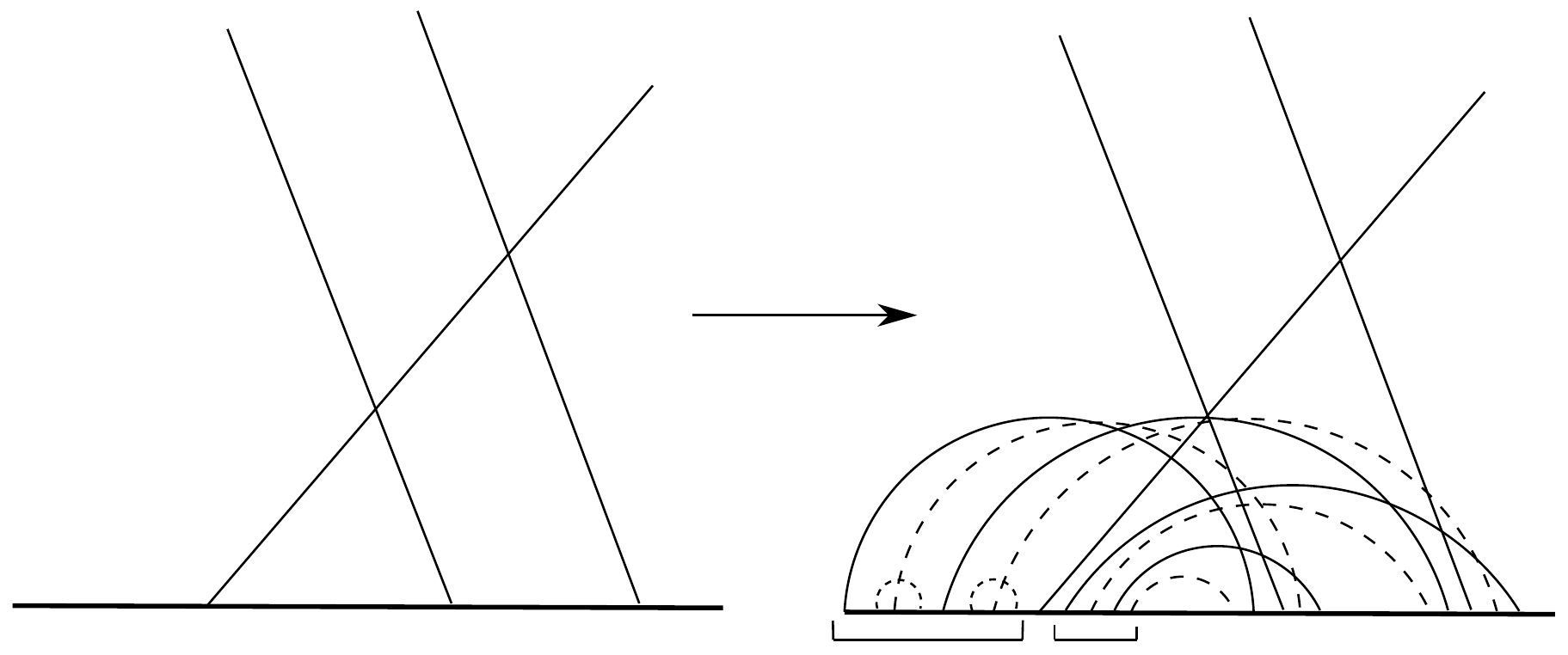}
\end{figure}

If $w_1$ and $w_2$ differ by the second kind of move in Lemma \ref{basept}, then the repeated conjugation gives us a type 3 and a type $3'$ move, and the $x_0$ symbol gives us a third type 3 or $3'$ move. After performing all of these moves, the remaining chords will cancel with type 2 and type $2'$ moves, and this reduces us back to the original diagram. 

In particular, we conjugate every symbol with index less than $k$ by $x_k$. This produces the chords for a type $3'$ move based just before $x_0[_k$ or $[_kx_0^{-1}$ and chords for a type 3 move based just after $x_0[_k$ or $[_kx_0^{-1}$. Also, the $x_0$ or $x_0^{-1}$ symbol produces chords for a type 3 or $3'$ move based where the symbol was located. All three of these moves, when applied, will produce chords which are parallel to the $k$-chord, except in the case when we have an $x_0^{-1}$ symbol, which gives a positive chord which crosses the $k$-chord. The only thing potentially preventing these chords from being adjacent to the $k$-chord is the chords from the $x_k^{-1}$ symbol at the end.  Thus, the chords from the type 3 and $3'$ moves from conjugation will cancel with each other by a type 2 move. If we had a $x_0$ and $x_k$ symbol, then the negative chord from the type $3'$ move generated by the $x_0$ symbol cancels with the $k$-chord, and we are done. If there was a $x_0^{-1}$ and $x_k^{-1}$ symbol, then the chords from the $x_k^{-1}$ symbol look like the chords from a type $2'$ move, except that they straddle the first endpoint of the $k$-chord, and the type 3 move gives us a positive chord which crosses the $k$-chord and the chords from the $x_k^{-1}$ symbol. At this point, a type 2 move cancels the $k$-chord with the positive chord from $x_k^{-1}$, and a type $2'$ move cancels the negative chords from $x_k^{-1}$ with the positive chord from the type 3 move from the $x_0^{-1}$ symbol. This takes us back to the original diagram.

Before we deal with the final case of Lemma \ref{braidmv} moves, we must develop even more moves which can be built out of type 1, 2, and 3 moves. We will call these moves slide moves and push moves respectively. Slide moves occur when we take a pair of chords which would cancel by a type 2 or type $2'$ move, except they straddle a collection of chords which all go to one side of the almost cancelable pair. We then apply a type 3 and $3'$ move to the chords of the almost cancelable pair, and then simplify the resulting diagram with several type 2 and $2'$ moves.  

\begin{figure}[h]
\caption{Example slide and push moves.}
\centering
\includegraphics[scale = 0.5]{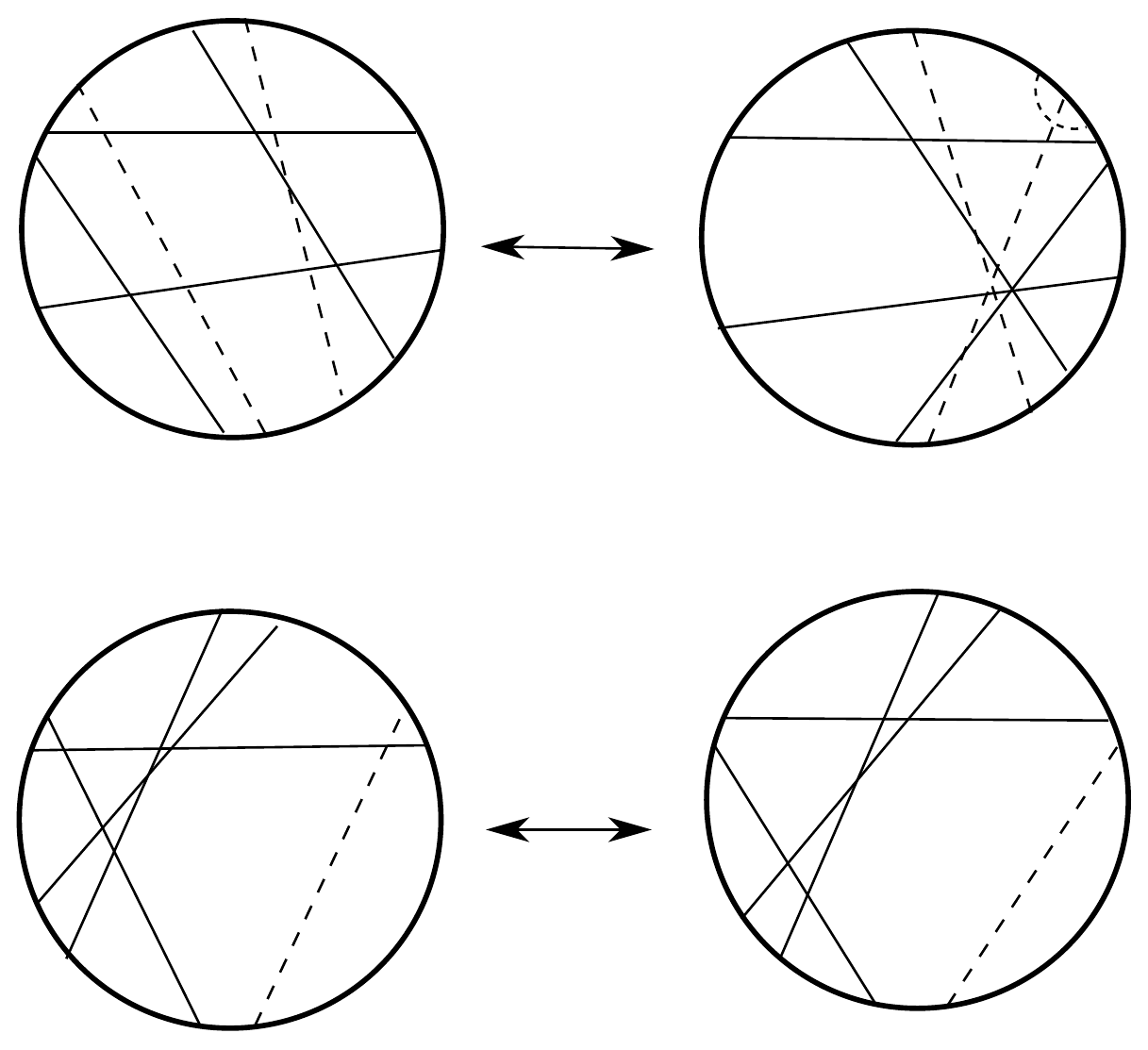}
\end{figure}

A push move is a special combination of two slide moves, which has the net effect of moving a chord inside of an almost cancellable pair. One slide move moves the pair away from the chord, and one moves it back. 

One should be careful when applying slide and push moves because sometimes other chords get in the way of the cancellations. We just use the terms ``slide'' and ``push'' as shorthand to describe the general technique. 

We now consider moves from Lemma \ref{braidmv}.

In the case of the first kind of move which only applies to one index, we have that every $x_i$ symbol is conjugated by $x_i$ which leaves it unchanged. After we delete all $x_i$ symbols that lie outside of their index, the result is that the only change from the original is that $[_i$ has an $x_i^{-1}$ symbol inserted after it and $]_i^{\pm}$ has an $x_i$ symbol inserted before it. This results in a set of chords in the diagram which can be transformed to a negative chord over the $[_i$ endpoint and all initial endpoints of chords from the $x_i$ symbols, and also a small positive chord around the $]_i^{\pm}$ endpoint. Now, we can repeatedly apply push moves to the negative chord to make it no longer cross any of the chords from the $x_i$ symbols. Finally, we can cancel the small positive and small negative chords with a type 2 move. 

We do not yet have a way to construct the second kind of moves from Lemma \ref{braidmv} with type 1, 2, and 3 moves. Therefore, we have invented braid moves to cover this case. Braid moves are precicely the transformations of $\Sigma(D)$ that can occur by transforming $D$ by the second kind of braid group generator in Lemma \ref{braidmv}.

\end{proof}

\begin{cnj}
Braid moves can be built from moves of type 1, 2, and 3.
\end{cnj}

\nocite{*}

\bibliography{Refrences}{}

\begin{thebibliography}{1}

\bibitem{braids}
J.~Birman.
\newblock {\em Braids, Links, and Mapping Class Groups}, volume~82 of {\em
  Annals of Mathematics Studies}.
\newblock Princeton University Press, 1974.

\bibitem{IntroToVass}
S.~Chmutov, S.~Duzhin, and J.~Mostovoy.
\newblock {\em Introduction to Vassiliev Knot Invariants}.
\newblock Cambridge University Press, 2012.

\bibitem{crom}
P.~Cromwell.
\newblock Arc presentations of knots and links.
\newblock {\em Knot Theory, Banach Center Publications}, 42:57--64, 1998.

\bibitem{knots}
L.~Kauffman.
\newblock {\em On Knots}, volume 115 of {\em Annals of Mathematics Studies}.
\newblock Princeton University Press, 1987.

\bibitem{kon}
M.~Kontsevich.
\newblock Vassiliev's knot invariants.
\newblock {\em Adv. Soviet Math.}, 16(2):137--150, 1993.

\bibitem{rolfsen}
D.~Rolfsen.
\newblock {\em Knots and Links}.
\newblock Publish or Perish Inc., 1976.

\end{thebibliography}
\bibliographystyle{plain}

\end{document}